\documentclass[a4paper,twoside]{article}
\usepackage{a4}
\usepackage{amssymb}
\usepackage{amsmath}
\usepackage{url}
\usepackage[active]{srcltx}
\usepackage[pagebackref,colorlinks,citecolor=blue,linkcolor=blue,urlcolor=blue]{hyperref}
\usepackage[dvipsnames]{color}
\usepackage{amsthm}
\usepackage{upref}
\allowdisplaybreaks[2] 
\newcount\minutes \newcount\hours
\hours=\time
\divide\hours 60
\minutes=\hours
\multiply\minutes -60
\advance\minutes \time
\newcommand{\klockan}{\the\hours:{\ifnum\minutes<10 0\fi}\the\minutes}
\newcommand{\tid}{\today\ \klockan}
\newcommand{\prtid}{\smash{\raise 10mm \hbox{\LaTeX ed \tid}}}
\renewcommand{\prtid}{}
\makeatletter
\def\sectionmark#1{} 
\def\subsectionmark#1{}
\newcommand{\sectnr}{\ifnum \c@secnumdepth >\z@
                 \thesection.\hskip 1em\relax \fi}
\def\@evenhead{\footnotesize\rm\thepage\hfil\leftmark\hfil\llap{\prtid}}
\def\@oddhead{\footnotesize\rm\rlap{\prtid}\hfil\rightmark\hfil\thepage}
\def\tableofcontents{\section*{Contents} 
 \@starttoc{toc}}
\makeatother
\makeatletter
\def\@biblabel#1{#1.}
\makeatother
\makeatletter
\let\Thebibliography=\thebibliography
\renewcommand{\thebibliography}[1]{\def\@mkboth##1##2{}\Thebibliography{#1}
\addcontentsline{toc}{section}{References}
\frenchspacing 
\setlength{\@topsep}{0pt}
\setlength{\itemsep}{0pt}%
\setlength{\parskip}{0pt plus 2pt}%
}
\makeatother
\makeatletter
\def\mdots@{\mathinner.\nonscript\!.%
 \ifx\next,.\else\ifx\next;.\else\ifx\next..\else
 \nonscript\!\mathinner.\fi\fi\fi}
\let\ldots\mdots@
\let\cdots\mdots@
\let\dotso\mdots@
\let\dotsb\mdots@
\let\dotsm\mdots@
\let\dotsc\mdots@
\def\vdots{\vbox{\baselineskip2.8\p@ \lineskiplimit\z@
    \kern6\p@\hbox{.}\hbox{.}\hbox{.}\kern3\p@}}
\def\ddots{\mathinner{\mkern1mu\raise8.6\p@\vbox{\kern7\p@\hbox{.}}%
    \raise5.8\p@\hbox{.}\raise3\p@\hbox{.}\mkern1mu}}
\makeatother
\makeatletter
\let\Enumerate=\enumerate
\renewcommand{\enumerate}{\Enumerate%
\setlength{\@topsep}{0pt}
\setlength{\itemsep}{0pt}%
\setlength{\parskip}{0pt plus 1pt}%
\renewcommand{\theenumi}{\textup{(\alph{enumi})}}%
\renewcommand{\labelenumi}{\theenumi}%
}
\let\endEnumerate=\endenumerate
\renewcommand{\endenumerate}{\endEnumerate\unskip}
\makeatother
\makeatletter
\def\@seccntformat#1{\csname the#1\endcsname.\quad}
\makeatother
\newcommand{\authortitle}[2]{\author{#1}\title{#2}\markboth{#1}{#2}}
\newcommand{\auth}[2]{{#1, #2.}}
\newcommand{\art}[6]{{\sc #1, \rm #2, \it #3 \bf #4 \rm (#5), \mbox{#6}.}}
\newcommand{\artin}[3]{{\sc #1, \rm #2,  in #3.}}
\newcommand{\artprep}[3]{{\sc #1, \rm #2, #3.}}

\newcommand{\book}[3]{{\sc #1, \it #2, \rm #3.}}
\newcommand{\AND}{{\rm and }}
\newcommand{\artnopt}[6]{{\sc #1, \rm #2, \it #3\/ \bf #4 \rm (#5), \mbox{#6}}}
\newcommand{\arXiv}[1]{{\tt \href{https://arxiv.org/abs/#1}{arXiv:#1}}}
\RequirePackage{amsthm}
\newtheoremstyle{descriptive}%
  {\topsep}   
  {\topsep}   
  {\rmfamily} 
  {}          
  {\bfseries} 
  {.}         
  { }         
  {}          
\newtheoremstyle{propositional}%
  {\topsep}   
  {\topsep}   
  {\itshape}  
  {}          
  {\bfseries} 
  {.}         
  { }         
  {}          
\newtheoremstyle{remarkstyle}%
  {\topsep}   
  {\topsep}   
  {\rmfamily}  
  {}          
  {\itshape} 
  {.}         
  { }         
  {}          
\theoremstyle{propositional}
\newtheorem{thm}{Theorem}[section]
\newtheorem{prop}[thm]{Proposition}
\newtheorem{lem}[thm]{Lemma}
\newtheorem{cor}[thm]{Corollary}
\theoremstyle{descriptive}
\newtheorem{deff}[thm]{Definition}
\newtheorem{example}[thm]{Example}
\newtheorem{remark}[thm]{Remark}
\makeatletter
\renewenvironment{proof}[1][\proofname]{\par
  \pushQED{\qed}%
  \normalfont
  \trivlist
  \item[\hskip\labelsep
        \itshape
    #1\@addpunct{.}]\ignorespaces
}{%
  \popQED\endtrivlist\@endpefalse
}
\makeatother
\newcommand{\setm}{\setminus}
\renewcommand{\emptyset}{\varnothing}
{\catcode`p =12 \catcode`t =12 \gdef\eeaa#1pt{#1}}      
\def\accentadjtext#1{\setbox0\hbox{$#1$}\kern   
                \expandafter\eeaa\the\fontdimen1\textfont1 \ht0 }
\def\accentadjscript#1{\setbox0\hbox{$#1$}\kern 
                \expandafter\eeaa\the\fontdimen1\scriptfont1 \ht0 }
\def\accentadjscriptscript#1{\setbox0\hbox{$#1$}\kern   
                \expandafter\eeaa\the\fontdimen1\scriptscriptfont1 \ht0 }
\def\accentadjtextback#1{\setbox0\hbox{$#1$}\kern       
                -\expandafter\eeaa\the\fontdimen1\textfont1 \ht0 }
\def\accentadjscriptback#1{\setbox0\hbox{$#1$}\kern     
                -\expandafter\eeaa\the\fontdimen1\scriptfont1 \ht0 }
\def\accentadjscriptscriptback#1{\setbox0\hbox{$#1$}\kern 
                -\expandafter\eeaa\the\fontdimen1\scriptscriptfont1 \ht0 }
\def\itoverline#1{{\mathsurround0pt\mathchoice
        {\rlap{$\accentadjtext{\displaystyle #1}
                \accentadjtext{\vrule height1.593pt}
                \overline{\phantom{\displaystyle #1}
                \accentadjtextback{\displaystyle #1}}$}{#1}}
        {\rlap{$\accentadjtext{\textstyle #1}
                \accentadjtext{\vrule height1.593pt}
                \overline{\phantom{\textstyle #1}
                \accentadjtextback{\textstyle #1}}$}{#1}}
        {\rlap{$\accentadjscript{\scriptstyle #1}
                \accentadjscript{\vrule height1.593pt}
                \overline{\phantom{\scriptstyle #1}
                \accentadjscriptback{\scriptstyle #1}}$}{#1}}
        {\rlap{$\accentadjscriptscript{\scriptscriptstyle #1}
                \accentadjscriptscript{\vrule height1.593pt}
                \overline{\phantom{\scriptscriptstyle #1}
                \accentadjscriptscriptback{\scriptscriptstyle #1}}$}{#1}}}}
\def\itunderline#1{{\mathsurround0pt\mathchoice
        {\rlap{$\underline{\phantom{\displaystyle #1}
                \accentadjtextback{\displaystyle #1}}$}{#1}}
        {\rlap{$\underline{\phantom{\textstyle #1}
                \accentadjtextback{\textstyle #1}}$}{#1}}
        {\rlap{$\underline{\phantom{\scriptstyle #1}
                \accentadjscriptback{\scriptstyle #1}}$}{#1}}
        {\rlap{$\underline{\phantom{\scriptscriptstyle #1}
                \accentadjscriptscriptback{\scriptscriptstyle #1}}$}{#1}}}}
\newcommand{\Cpt}{{C_{s,p}}}
\newcommand{\Csp}{{C_{s,p}}}
\newcommand{\cpt}{{\capp_{s,p}}}
\newcommand{\cpa}{{\capp_{p,|t|^a}}}
\newcommand{\csp}{{\capp_{s,p}}}
\DeclareMathOperator{\diam}{diam}
\DeclareMathOperator{\dist}{dist}
\DeclareMathOperator{\Outer}{outer}
\DeclareMathOperator{\Div}{div}
\DeclareMathOperator{\capp}{cap}
\newcommand{\grad}{\nabla}
\DeclareMathOperator{\Lip}{Lip}
\newcommand{\Lipc}{{\Lip_c}}
\DeclareMathOperator*{\essliminf}{ess\,lim\,inf}
\newcommand{\bdry}{\partial}
\newcommand{\bdy}{\bdry}
\newcommand{\de}{\delta}
\newcommand{\eps}{\varepsilon}
\newcommand{\la}{\lambda}
\newcommand{\Om}{\Omega}
\newcommand{\gt}{\tilde{g}}
\newcommand{\uhat}{\hat{u}}
\newcommand{\ub}{\bar{u}}
\newcommand{\vb}{\bar{v}}
\newcommand{\ut}{\tilde{u}}
\renewcommand{\phi}{\varphi}
\newcommand{\p}{{$p\mspace{1mu}$}}
\newcommand{\R}{\mathbf{R}}
\newcommand{\Rn}{{\R^n}}
\newcommand{\eR}{{\overline{\R}}}
\def\cprime{{\mathsurround0pt$'$}}
\newcommand{\limplus}{{\mathchoice{\vcenter{\hbox{$\scriptstyle +$}}}
  {\vcenter{\hbox{$\scriptstyle +$}}}
  {\vcenter{\hbox{$\scriptscriptstyle +$}}}
  {\vcenter{\hbox{$\scriptscriptstyle +$}}}
}}
\newcommand{\limminus}{{\mathchoice{\vcenter{\hbox{$\scriptstyle -$}}}
  {\vcenter{\hbox{$\scriptstyle -$}}}
  {\vcenter{\hbox{$\scriptscriptstyle -$}}}
  {\vcenter{\hbox{$\scriptscriptstyle -$}}}
}}
\newcommand{\clB}{\itoverline{B}}
\newcommand{\clOm}{\overline{\Om}}
\newcommand{\LL}{\mathcal{L}}
\newcommand{\UU}{\mathcal{U}}
\newcommand{\UUt}{\widetilde{\UU}}
\newcommand{\uP}{\itoverline{P}}     
\newcommand{\lP}{\itunderline{P}} 
\newcommand{\uQ}{\itoverline{Q}}     
\newcommand{\vt}{\tilde{v}}
\DeclareMathOperator{\supp}{supp}
\newcommand{\E}{\mathcal{E}}
\newcommand{\K}{\mathcal{K}}
\newcommand{\clG}{\itoverline{G}}
\newcommand{\Vsp}{V^{s,p}}
\newcommand{\VspOm}{V^{s,p}(\Om)}
\newcommand{\Vspo}{V^{s,p}_0}
\newcommand{\VspoOm}{V^{s,p}_0(\Om)}
\newcommand{\Wsp}{W^{s,p}}
\newcommand{\pot}{\mathfrak{R}}
\newcommand{\pothat}{\widehat{\pot}}
\newcommand{\La}{\Lambda}
\newcommand{\Omc}{{\Om^c}}
\newcommand{\Gc}{G^c}
\newcommand{\Ec}{E^c}
\newcommand{\Et}{\widetilde{\E}}
\newcommand{\Kt}{\widetilde{\K}}
\newcommand{\Ht}{\widetilde{H}}
\newcommand{\A}{\mathcal{A}}

\newcounter{saveenumi}
\numberwithin{equation}{section}
\newcommand{\eqv}{\ensuremath{
\mathchoice{\quad \Longleftrightarrow \quad}{\Leftrightarrow}
                {\Leftrightarrow}{\Leftrightarrow}}}
\newcommand{\imp}{\ensuremath{
\mathchoice{\quad \Longrightarrow \quad}{\Rightarrow}
                {\Rightarrow}{\Rightarrow}}}
\newenvironment{ack}{\medskip{\it Acknowledgement.}}{}

\begin{document}

\authortitle{Anders Bj\"orn, Jana Bj\"orn and Minhyun Kim}
            {Perron solutions and boundary regularity
 for nonlocal nonlinear Dirichlet problems}

 \author{
 Anders Bj\"orn \\
 \it\small Department of Mathematics, Link\"oping University, SE-581 83 Link\"oping, Sweden\\
 \it\small and Theoretical Sciences Visiting Program, Okinawa Institute of\\
 \it\small
  Science and Technology Graduate University, Onna, 904-0495, Japan \\
 \it \small anders.bjorn@liu.se, ORCID\/\textup{:} 0000-0002-9677-8321
 \\
 \\
 Jana Bj\"orn \\
 \it\small Department of Mathematics, Link\"oping University, SE-581 83 Link\"oping, Sweden\\
 \it\small and Theoretical Sciences Visiting Program, Okinawa Institute of\\
 \it\small
  Science and Technology Graduate University, Onna, 904-0495, Japan \\
 \it \small jana.bjorn@liu.se, ORCID\/\textup{:} 0000-0002-1238-6751
 \\
 \\
 Minhyun Kim \\
 \it\small Department of Mathematics \& Research Institute for Natural Sciences, \\
 \it\small Hanyang University, 04763 Seoul, Republic of Korea \\
 \it \small minhyun@hanyang.ac.kr, ORCID\/\textup{:} 0000-0003-3679-1775
 }

\date{Preliminary version, \today}
\date{}

\maketitle

\noindent{\small
{\bf Abstract}.}
For nonlinear operators of fractional \p-Laplace type, we consider two types of solutions
to the nonlocal Dirichlet problem: Sobolev solutions based on fractional Sobolev spaces 
and Perron solutions based on superharmonic functions.
These solutions give rise to two different concepts of regularity for boundary points, 
namely Sobolev and Perron regularity. 
We show that these two notions are equivalent
and we also provide several characterizations of regular boundary points. 
Along the way, we give a new definition of Perron solutions, 
which is applicable to arbitrary exterior Dirichlet
data $g: \Omega^c \to [-\infty,\infty]$.
We obtain resolutivity 
results for these Perron solutions,
and show that the Sobolev  and Perron solutions coincide 
for a large class of exterior Dirichlet data.
This also implies invariance of the Perron solutions under perturbations
on sets of zero fractional capacity.
A uniqueness result for the Dirichlet problem is also obtained for 
the class of bounded  solutions
taking prescribed continuous 
exterior data quasieverywhere on the boundary.

\medskip

\noindent {\small \emph{Key words and phrases}:
Barrier,
boundary regularity,
capacity,
fractional \p-Laplacian,
nonlocal nonlinear Dirichlet problem,
Perron solution,
resolutivity,
Sobolev solution.
}

\medskip

\noindent {\small \emph{Mathematics Subject Classification} (2020):
Primary:
35R11. 
Secondary:  
31C15, 
31C45, 
35J66. 
}

\medskip

\noindent {\small \emph{Funding}: 
AB and JB were supported by the Swedish Research Council,
  grants 2018-04106, 2020-04011 and 2022-04048.
MK was supported by the National Research Foundation of Korea (NRF) 
grant funded by the Korean government (MSIT) (RS-2023-00252297).
}

\medskip

\noindent
{\small {\bf Declarations}. \\
\emph{Conflicts of interest}:   None.
\\
\emph{Availability of data and material}:   Not applicable.
\\
\emph{Code availability}:   Not applicable.
    } 

\section{Introduction}

The aim of this paper is 
to study the Perron method and boundary regularity for 
the nonlocal nonlinear equation $\LL u =0$.
The operator $\LL$ is of the form
\begin{equation*}
\mathcal{L}u(x) = 2 \,\mathrm{p.v.} \int_{\R^n} |u(x)-u(y)|^{p-2} (u(x)-u(y)) k(x, y) \,dy,
\end{equation*}
where $1<p<\infty$, $n \ge 1$ and $k: \R^n \times \R^n \to [0, \infty]$ is a symmetric measurable kernel
that satisfies the ellipticity condition
\begin{equation} \label{eq-comp-(x,y)}
\frac{\Lambda^{-1}}{|x-y|^{n+sp}} \leq k(x, y) \leq \frac{\Lambda}{|x-y|^{n+sp}},
\end{equation}
with $0<s<1$ and $\La \ge1$.
Note that $\LL$ is the fractional \p-Laplacian $(-\Delta_p)^s$ when 
$k(x,y)=|x-y|^{-n-sp}$.

Nonlocal and fractional problems have been studied for a long time,
see e.g.\ \cite{AH,BH86}, but in recent years, starting perhaps
with the pioneering paper by Caffarelli--Silvestre~\cite{CafSil},
they have received
a lot more attention (see e.g.\ our list of references).
Fractional operators also have significant applications in various contexts,
such as 
the obstacle problem \cite{CF13,Sil07}, 
materials science \cite{Bat06,BC99}, 
phase transitions \cite{CSM05,FV11,SV09}, 
soft thin films \cite{Kur06}, quasi-geostrophic dynamics \cite{CV10} 
and image processing \cite{GO08}.

We assume in the rest of the introduction 
that  $\Om \subset \Rn$ is a nonempty bounded open set.
In the \emph{Dirichlet problem} one seeks 
a solution of $\LL u=0$
in $\Om$ which takes on some
prescribed exterior Dirichlet data 
$g: \Omc \to \eR:=[-\infty,\infty]$. 
Due to the nonlocal nature of the problem the Dirichlet data
have to be prescribed on the whole complement $\Omc$ of $\Om$.
For this reason, the Dirichlet problem is 
also called the exterior value problem or the complement value problem, 
but we will call it the Dirichlet problem for short.

As  in the classical setting of harmonic functions one can in general
not solve the Dirichlet problem so that the prescribed boundary values 
are taken as limits on $\bdy \Om$,
even if $g$ is bounded and continuous.
Hence, the Dirichlet problem has to be understood in a relaxed sense.
The two main ways of doing so are the
Sobolev solutions (taking the boundary values in the Sobolev sense) 
and the Perron solutions.

A natural question is then which boundary points are \emph{regular}, 
i.e.\ whether all
continuous Dirichlet data are  attained as limits of the 
corresponding solutions.
The two notions of solutions lead to two different types of boundary regularity:
Sobolev regularity and Perron regularity.
One of our main results is that 
these two types of regularity are equivalent.
To show this, we  first need to 
substantially develop the Perron method.
For classical harmonic functions it
goes back to 
Perron~\cite{Per23} and independently Remak~\cite{remak}, see
the end of the introduction for further historical remarks.
Our results on boundary regularity are summarized in Theorem~\ref{thm-main-reg}.

To the best of our knowledge, the solvability of the nonlocal Dirichlet problem in
the Sobolev sense was first established by Hoh--Jacob~\cite{HJ96} when $p=2$. 
It was later extended to the nonlocal Dirichlet problems associated with
(not necessarily symmetric)
kernels $k$ by Felsinger--Kassmann--Voigt~\cite{FKV15}. 
For general $p$ and operators $\LL$ as above, 
it was solved using the direct method of calculus of variations 
by Di Castro--Kuusi--Palatucci~\cite{DCKP16}. 
These results together with the interior regularity results prove the existence and 
uniqueness of an $\LL$-harmonic function (a weak solution that is continuous in $\Om$) 
solving the nonlocal Dirichlet problem with Sobolev data $g \in \Vsp(\Om)$, 
see Section~\ref{sec-D-obstacle}.
We denote this solution by $Hg$.

Perron solutions of  nonlocal equations were studied by Bliedtner--Hansen~\cite{BH86} 
for the fractional linear 
Laplacian $(-\Delta)^s$, Lindgren--Lindqvist~\cite{LL17} for 
the fractional nonlinear \p-Laplacian $(-\Delta_p)^s$ 
and by Korvenp\"a\"a--Kuusi--Palatucci~\cite{KKP17} for general $\LL$ as here. 
However, the definitions of Perron solutions in these works are all different 
even when restricted to $(-\Delta)^s$, 
and the one in~\cite{KKP17} seems unnaturally restrictive. 
Before discussing the differences between these definitions, 
we provide our definition of Perron solution. 
For $\LL$-superharmonic functions we follow the definition
in~\cite{KKP17}, see Definition~\ref{def:superharmonic}.

\begin{deff} \label{def-Perron}
Let $g: \Omc \to \eR$. 
The \emph{upper class}  $\mathcal{U}_{g}=\mathcal{U}_{g}(\Om)$ 
of $g$ consists of all functions 
$u: \R^{n} \to (-\infty,\infty]$ such that
\begin{enumerate}
\item
$u$ is $\LL$-superharmonic in $\Omega$,
\setcounter{saveenumi}{\value{enumi}}
\item \label{P-b}
$u$ is bounded from below in $\Om$, 
\item \label{P-c}
$\liminf_{\Om \ni y \to x} u(y) \geq g(x)$ for all $x \in \partial \Omega$,
\item \label{P-d}
$u \ge g$ in $\Omc$.
\end{enumerate}

The \emph{upper Perron solution} of $g$ is defined by
\[ 
    \uP g (x)= \uP_\Om g (x) = \inf_{u \in \UU_g}  u(x), \quad x \in \R^n,
\]
and the \emph{lower Perron solution} of $g$ by 
$\lP g = - \uP (-g)$.
(As usual, $\inf \emptyset = \infty$.)
\end{deff}

We usually omit $\Om$ from the notation.
When $\lP g = \uP g$  we denote the common solution by $Pg$ and say that
$g$ is \emph{resolutive}.

Our definition is essentially the same
as the one given in Lindgren--Lindqvist~\cite[Section~4]{LL17}.
(In \cite{LL17} they only considered bounded continuous functions $g$ and operators 
without kernels.)
On the other hand, it differs from the definition
in Korvenp\"a\"a--Kuusi--Palatucci~\cite[Definition~2]{KKP17} in 
conditions \ref{P-c} and~\ref{P-d}.
In particular, we have an inequality in \ref{P-d}, 
which makes it possible to define Perron solutions 
for arbitrary data $g$ and not
only for $g \in L^{p-1}_{sp}(\R^n)$ as in~\cite{KKP17}.
More importantly, it immediately gives the comparisons
\begin{equation} \label{eq-Pg1-Pg2}
    \lP g_1 \le \lP g_2 
\quad \text{and} \quad 
    \uP g_1 \le \uP g_2 \quad  \text{in } \R^n
\qquad \text{whenever } g_1 \le g_2 \text{ in } \Omega^c,
\end{equation}
which do not appear in~\cite{KKP17}.
Nevertheless, we are able to show that the Perron solutions
used in~\cite{KKP17} essentially coincide with our Perron solutions whenever
they are defined, 
see Theorem~\ref{thm-Perron-meas}.
In particular they do satisfy~\eqref{eq-Pg1-Pg2}.

The analogue of \ref{P-d} in~\cite{KKP17}
is that $u = g$ a.e.\ in $\Omc$. 
Since one can freely modify $u(x)$  arbitrarily (with a value in 
$(-\infty,\infty]$) for any particular $x \in \Omc$, 
one sees that with their definition
either $\uP g \equiv -\infty$ outside $\Om$ or $\uP g \equiv \infty$ in $\Rn$, 
contradicting some of their statements.
This problem disappears if their condition (iv) is replaced by 
requiring that $u\equiv g$  outside $\Omc$.
Note that the infimum defining $\uP g$ is typically 
taken over an uncountable collection.

In the earlier work by Bliedtner--Hansen~\cite{BH86},
a different upper class was introduced
for the linear case $p=2$.
They used hyperharmonic functions
which are lower semicontinuous in $\R^n$
and lower $\mathcal{P}$-bounded in $\R^n$
instead of superharmonic functions
which are bounded from below in $\Om$
and lower semicontinuous only  in $\Om$.
(Hyperharmonic functions and lower $\mathcal{P}$-boundedness
will not be discussed here. See \cite[p.~343]{BH86}.)
They obtained the natural properties such as
\eqref{eq-Pg1-Pg2} and 
\begin{equation} \label{eq-lP<=uP}
    \lP g \le \uP g 
\end{equation}
and studied Perron regularity with respect to their Perron solutions.
Sobolev spaces, Sobolev  
solutions and Sobolev regularity were not considered 
in~\cite{BH86}.
For our operator $\LL$, the fundamental inequality~\eqref{eq-lP<=uP}
follows from the comparison principle in Theorem~\ref{thm-comp},
which is due to  Korvenp\"a\"a--Kuusi--Palatucci~\cite[Theorem~16]{KKP17}.

In view of our results below,
Definition~\ref{def-Perron} of Perron solutions seems natural.
We now turn to our main results on Perron solutions,
which substantially extend the nonlocal nonlinear theory 
initiated in the
two papers \cite{KKP17,LL17}. 
Here, the tail space $L^{p-1}_{sp}(\Omega^c)$ is defined by replacing 
$\Rn$ in \eqref{eq-tail} by $\Omega^c$.
By an \emph{$\LL$-harmonic function} 
in $\Om$ we mean a weak solution that is continuous in $\Om$.

\begin{thm}\label{thm-harmonicity}
Let $g : \Omc \to \eR$.
\begin{enumerate}
\item \label{thm-harmonicity-a}
Then $\uP g$ is either continuous in $\Om$ or identically\/ $\pm\infty$ in $\Om$.
\item \label{thm-harmonicity-b}
If $g \in L^{p-1}_{sp}(\Omega^c)$,
then $\uP g$ is either $\LL$-harmonic in $\Om$ or identically\/ $\pm\infty$ in $\Om$.
\end{enumerate}
\end{thm}

Theorem~\ref{thm-harmonicity} 
(which follows from Theorems~\ref{thm-Perron} and~\ref{thm-Perron-meas})
differs in two aspects from the corresponding 
nonlinear local case for e.g.\ the local \p-Laplace equation $-\Delta_pu=0$.
First of all, $\Omega$ is not necessarily connected in Theorem~\ref{thm-harmonicity} whereas 
similar statements for local equations only apply
to each connected  component of $\Om$ separately.
Second, the upper and the lower Perron solutions 
of the local \p-Laplace equation
are always \p-harmonic (or identically $\pm \infty$) in a domain 
for arbitrary (even nonmeasurable) boundary data. 
However, for nonmeasurable exterior data  $g$,  
the continuity of $\uP g$ and $\lP g$ is the best we can obtain in the nonlocal case. 
Indeed, for any particular $x \in \Omc$ and $u \in \UU_g$, 
we can freely modify
$u(x)$ as long as $g(x) \le u(x) \ne -\infty$. 
Hence, 
\begin{equation} \label{eq-uP=g}
\text{either} \quad 
\uP g=g \text{ in } \Omc 
\quad \text{or} \quad 
\uP g \equiv \infty  \text{ in } \Rn,
\end{equation}
depending on whether $\UU_g\neq \emptyset$ or not.
In particular,
if $g$ is nonmeasurable, then $\uP g$ cannot be $\LL$-harmonic  in $\Om$
because $\LL$-harmonicity requires measurability in $\Rn$, 
and thus in $\Om^c$.

We refer the reader to Lindgren--Lindqvist~\cite[Theorem~22]{LL17} 
and Korvenp\"a\"a--Kuusi--Palatucci~\cite[Theorem~2]{KKP17} 
for earlier results similar to the second assertion in Theorem~\ref{thm-harmonicity}.

The next main results concern resolutivity, 
which is one of the central questions for the Perron method.
It also relates Perron solutions to the Sobolev solutions $Hg$,
which is a key step to showing  
that Sobolev and Perron regularity are equivalent.
Here, $\VspOm$ is a suitable fractional Sobolev space associated with the
operator $\LL$, see \eqref{eq-def-Vsp} for the definition.

\begin{thm}\label{thm-main-Hg=Pg}
Let  $g \in C(\R^n)\cap \VspOm$
and $h:\Rn \to \eR$ be such that 
$h=0$ in $\Ec$, where the Sobolev capacity $\Csp(E)=0$. 
Then both $g$ and $g+h$ are  resolutive and 
\[
   H(g+h)=P(g+h)=Hg+ h\chi_{\Omc} = Pg+ h\chi_{\Omc}.
\]
Moreover, both $Pg$ and $P(g+h)$ are $\LL$-harmonic in $\Om$.
\end{thm}

\begin{thm} \label{thm-Perron-res-cont}
Let $g:\Omc\to\R$ be a bounded measurable function that is continuous 
at every $x\in\bdy\Om$.
Assume that $h:\Omc \to \eR$ is such that   
$h=0$ in $\Omc \setm E$, where the Sobolev capacity $\Csp(E)=0$. 
Then both $g$ and $g+h$ are resolutive and 
\[
     P(g+h) =  Pg + h\chi_{\Omc}.   
\]
Moreover, both $Pg$ and $P(g+h)$ are $\LL$-harmonic in $\Om$.
\end{thm}

Lindgren--Lindqvist~\cite{LL17} 
proved the resolutivity (but not invariance 
with $h \not\equiv 0$)
of bounded continuous Dirichlet data $g$ (with a limit at infinity)
for $\LL= (-\Delta_p)^s$  (without kernel but with bounded right-hand sides).
We only require continuity at all $x \in \bdy \Om$ (and boundedness in $\Omc$).

In the linear nonlocal case, Bliedtner--Hansen~\cite{BH86} 
(using a different type of Perron solutions)  obtained
more extensive resolutivity results corresponding
to Brelot's classical resolutivity result for harmonic functions,
 but they substantially used linearity 
(which is not at our disposal here).
For example, the invariance under negligible perturbations of the 
Dirichlet data is an easy 
consequence of the linearity.
In neither paper, there are any results connecting Sobolev and Perron solutions.

In order to prove Theorem~\ref{thm-main-Hg=Pg} we show the Kellogg property, which is of independent interest.

\begin{thm}\label{thm-kellogg}
\textup{(Kellogg property)}
Let $I_\Om\subset\bdy\Om$ be the set of Sobolev irregular boundary points for 
$\LL u=0$
  in $\Om$.
  Then the Sobolev capacity $\Csp(I_\Om)=0$.
\end{thm}

Together with the invariance in Theorem~\ref{thm-Perron-res-cont},
the Kellogg property leads to the following existence and
uniqueness result, which substantially extends
Lemma~19 in Korvenp\"a\"a--Kuusi--Palatucci~\cite{KKP17}.
In Proposition~\ref{prop-Wsp-cpt} we give another application
of the Kellogg property.

\begin{thm} \label{BBS2-thm-intro}
Let $g:\Omc\to\R$ be a bounded measurable function that is
 continuous at every $x\in\bdy\Om$.
Then there exists a unique bounded $\LL$-harmonic function $u$ in $\Om$
such that $u \equiv g$ in $\Om^c$   
and 
\begin{equation} \label{eq-BBS2-thm}
         \lim_{\Om \ni x \to x_0} u(x) = g(x_0)
         \quad \text{for quasievery } x_0 \in \bdy \Om,
\end{equation}
i.e.\ outside a set $E$ with $\Csp(E)=0$.
Moreover $u=P g$.
\end{thm}

Finally we are ready to state our last main result, concerning boundary regularity. 
This follows from Theorems~\ref{thm-Wiener}, \ref{thm-Sobolev-reg}
and~\ref{thm-main-reg-Perron}.

\begin{thm}\label{thm-main-reg}
Let $x_0 \in \partial\Om$. Then the following are equivalent\/\textup{:}
\begin{enumerate}
\item \label{i-Sobolev}
$x_0$ is Sobolev regular, i.e.
\begin{equation*}
\lim_{\Om \ni x \to x_0} Hg(x) = g(x_0)
\quad \text{for every } g \in \VspOm \cap C(\R^n).
\end{equation*}
\item \label{i-Perron}
$x_0$ is Perron regular, i.e.
\begin{equation*}
\lim_{\Om \ni x \to x_0} Pg(x) = g(x_0)
\end{equation*}
for every bounded $g \in C(\Omc)$. 
\item \label{i-d}
\begin{equation*}
\lim_{\Om \ni x \to x_0} Hd_{x_0}(x)=0, 
\quad \text{where }   d_{x_0}(x):=\min\{1,|x-x_0|\}.
\end{equation*}
\item \label{i-contx0}
  \[
    \lim_{\Om \ni x \to x_0} \uP g(x)=g(x_0)
  \]  
for every bounded $ g:\Omc\to \R$ that
is continuous at $x_0$.
\item \label{i-Wiener}
\textup{(}The Wiener criterion\/\textup{)}
\begin{equation*}
\int_0^1 \biggl( \frac{\csp(\itoverline{B(x_0, \rho)}\setm \Om,B(x_0,2\rho))}{\rho^{n-sp}} \biggr)^{1/(p-1)} \, \frac{d\rho}{\rho} = \infty,
\end{equation*}
where $\csp$ is the condenser capacity.
\item \label{i-barrier}
There is a barrier at $x_0$.
\end{enumerate}
\end{thm}

Some remarks are in order. The coincidence of Sobolev and Perron regularity (the equivalence \ref{i-Sobolev}\eqv\ref{i-Perron}) is new even for the simplest case when the operator is given by $(-\Delta)^s$. 
Due to this equivalence one  does not have to distinguish between
Sobolev and Perron regularity. 
Condition~\ref{i-d} provides a criterion for the regularity that can be easily checked. The equivalence \ref{i-Perron}\eqv\ref{i-contx0} shows that one can characterize regular boundary points in terms of upper Perron solutions for a very general class of exterior data $g$, including nonmeasurable $g$. 
Recall that $\uP g$ is not $\LL$-harmonic in $\Om$ when $g$ is nonmeasurable,
but in view of \ref{i-contx0} it is of interest to also include
nonmeasurable boundary data $g$. 

The important Wiener criterion (i.e.\ the equivalence \ref{i-Sobolev}\eqv\ref{i-Wiener})
was recently obtained by 
Kim--Lee--Lee~\cite{KLL23} for general $\LL$ as here, and 
independently by Bj\"orn~\cite{JBWien} 
for $(-\Delta)^s$ 
using the Caffarelli--Silvestre extension~\cite{CafSil}.
It serves as a starting point for our study of boundary regularity.
The barrier characterization \ref{i-barrier}
 for Perron regularity was shown by Lindgren--Lindqvist~\cite{LL17}
 for $\LL= (-\Delta_p)^s$  (without kernel).
See Theorems~\ref{thm-Sobolev-reg} and~\ref{thm-main-reg-Perron} for further characterizations of 
regular boundary points.
In particular, we show in Theorem~\ref{thm-main-reg-Perron} that boundary regularity
for obstacle problems is also equivalent to Perron and Sobolev regularity.

Next we give some historical remarks.
For classical harmonic functions, the Perron method (or PWB method) was developed 
by Perron~\cite{Per23}, Remak~\cite{remak}, Wiener~\cite{Wie24a,Wie24b} and Brelot~\cite{Bre39}. 
It plays a fundamental role in potential theory
and has later been extended to wide classes of problems.

As far as we know, Aronsson~\cite[Section~4]{aronsson67} was
the first who used the Perron method in connection with a nonlinear
equation, namely for the $\infty$-Laplacian (also called
Aronsson's equation).
In connection with \p-harmonic functions (for $1<p<\infty$)
it was first used by
Granlund--Lindqvist--Martio~\cite{GLM86},
see also Kilpel\"ainen~\cite{Kilp89} and 
Heinonen--Kilpel\"ainen--Martio~\cite{HeKiMa}.
In the nonlinear local setting, invariance results 
such as in Theorems~\ref{thm-main-Hg=Pg} and~\ref{thm-Perron-res-cont}
(with $h \not\equiv 0$)
and a uniqueness result similar to
Theorem~\ref{BBS2-thm-intro} were 
obtained
by Bj\"orn--Bj\"orn--Shanmugalingam~\cite{BBS2},
see also  Bj\"orn--Bj\"orn--Mwasa~\cite{BBMwasa}.

Boundary regularity has
been studied
in depth for nonlinear local equations such as the 
\p-Laplace equation $-\Delta_pu=0$.
The classical Wiener criterion~\cite{Wie24b} for harmonic functions
was extended to the \p-Laplace equation
by Maz{\cprime}ya~\cite{mazya70} (sufficiency) and
Lindqvist--Martio~\cite{LM85} and Kilpel\"ainen--Mal\'y~\cite{KM94} (necessity).
The Kellogg property was shown by Kilpel\"ainen~\cite{Kilp89}.
In an equivalent form,
it 
also follows from earlier
results by Hedberg~\cite{Hedberg72} (for $p>2-1/n$)
and Hedberg--Wolff~\cite{HedbergWolff} (for $p \le 2-1/n$), together with 
Maz{\cprime}ya's sufficiency part of the Wiener criterion.
Barrier characterizations were obtained by 
Granlund--Lindqvist--Martio~\cite{GLM86}, Lehtola~\cite{Leh86} and Kilpel\"ainen~\cite{Kilp89}.
We refer the reader to Heinonen--Kilpel\"ainen--Martio~\cite{HeKiMa} 
and Bj\"orn--Bj\"orn~\cite{BBbook}
for further characterizations, historical remarks
and other topics in nonlinear potential theory associated with \p-harmonic functions.

The paper is organized as follows. In Section~\ref{sec-sobolev} we recall some function spaces and study basic properties of (weak) solutions of $\LL u=0$ in $\Om$. 
We also obtain some basic results for the fractional Sobolev space $\Vspo(\Om)$
which seem to be less trivial than one might expect.

Some results about the Dirichlet and obstacle problems associated with $\LL$ 
are presented in Section~\ref{sec-D-obstacle}. Our first characterizations of 
Sobolev regularity are given in Section~\ref{sec-Sob-reg}. 
The equivalence \ref{i-Sobolev}\eqv\ref{i-d} in Theorem~\ref{thm-main-reg} is proved here, 
as well as other
characterizations 
including some for
potentials.

Two capacities, namely the condenser capacity $\csp$ and 
the Sobolev capacity $\Csp$, 
play a fundamental role
in describing boundary regularity.
Section~\ref{sec-capacity} is devoted to the investigation of the close relationship
 between these capacities. With these tools at hand we prove 
the Kellogg property (Theorem~\ref{thm-kellogg}) in Section~\ref{sec-kellogg}. 

In Section~\ref{sec-obstacle} we obtain a convergence result for solutions of 
obstacle problems,
which plays a fundamental role when deducing 
Theorem~\ref{thm-main-Hg=Pg} in Section~\ref{sec-perron}. 
This convergence result (Theorem~\ref{thm-conv-obst-prob}) 
may be of independent interest.
Before turning to Perron solutions, 
we 
collect some results on 
$\LL$-superharmonic functions in Section~\ref{sec-superharmonic}.

Finally, in Section~\ref{sec-perron} we turn to Perron solutions.
Theorems~\ref{thm-harmonicity}--\ref{thm-Perron-res-cont} 
and more detailed results are proved in this section.
The full characterizations in Theorem~\ref{thm-main-reg} are deduced 
in Section~\ref{sec-perron-reg}. 
Here more characterizations are provided, including
some for obstacle problems.
Also Theorem~\ref{BBS2-thm-intro} is obtained in this section.

\begin{ack}
AB and JB were supported by the Swedish Research Council,
  grants 2018-04106, 2020-04011 and 2022-04048.
MK was supported by the National Research Foundation of Korea (NRF) 
grant funded by the Korean government (MSIT) (RS-2023-00252297).

The discussions leading to this paper started 
when all three authors
visited Institut Mittag-Leffler in the autumn of 2022 during the programme
\emph{Geometric Aspects of Nonlinear Partial Differential Equations}.
More progress was made during two visits of MK to  Link\"oping University 
and one, supported by SVeFUM, by AB and JB to Seoul in 2023.
Part of this research was conducted when all three authors 
visited Okinawa Institute of Science and Technology (OIST)
in 2024, AB and JB through the 
Theoretical Sciences Visiting Program (TSVP),
and when AB and JB visited Hanyang University, Seoul, in 2024,
supported by the NRF grant mentioned above.
We thank all these institutions  for their hospitality
and support.
\end{ack}

\section{The nonlocal equation and Sobolev spaces}\label{sec-sobolev}

\emph{Throughout the paper, we assume that 
$1<p<\infty$, $0<s<1$, $n \ge 1$, that $k$
 is a symmetric measurable kernel satisfying \eqref{eq-comp-(x,y)},
and that $\Om \subset \Rn$
is a nonempty  open set.
From  Section~\ref{sec-D-obstacle} onwards we also
always assume that $\Om$ is bounded.}

\medskip

In order to consider weak solutions of the equation
\begin{equation}\label{eq-Lu=0}
\LL u=0 \quad\text{in}~\Om,
\end{equation}
we need to define some function spaces.

For a measurable function $u: \Om \to \eR:=[-\infty,\infty]$ (which is
finite a.e.) we define the fractional seminorm  
\begin{equation*}
  [u]_{\Wsp(\Om)}=
\biggl(  \int_{\Omega}\int_{\Omega} \frac{|u(x)-u(y)|^p}{|x-y|^{n+s p}} 
      \, dy\, dx\biggr)^{1/p}. 
\end{equation*}
The fractional Sobolev space $\Wsp(\Omega)$ consists of the functions $u$ such that
the norm  
\begin{equation*}
\|u\|_{\Wsp(\Om)}^p:=\|u\|_{L^p(\Om)}^p+[u]_{\Wsp(\Om)}^p < \infty.
\end{equation*}
The above spaces and (semi)norms go under various names, 
such as fractional Sobolev, 
Gagliardo--Nirenberg, Sobolev--Slobodetski\u{\i} and Besov.

By $W^{s, p}_{\mathrm{loc}}(\Omega)$ we denote the space of functions that belong to $W^{s, p}(G)$ 
for every open $G \Subset \Omega$. 
As usual, by $E \Subset \Om$ we mean that $\itoverline{E}$
is a compact subset of $\Om$.
We refer the reader to Di Nezza--Palatucci--Valdinoci~\cite{DNPV12} for properties of these spaces.

As \eqref{eq-Lu=0} is a nonlocal equation, we also need a function space that captures integrability of
functions in the whole of $\Rn$.
The \emph{tail space} $L^{p-1}_{sp}(\R^n)$ is given by
\begin{equation}\label{eq-tail}
L^{p-1}_{sp}(\R^n) = \biggl\lbrace u \text{ measurable}: 
  \int_{\R^n} \frac{|u(y)|^{p-1}}{(1+|y|)^{n+sp}} \,dy < \infty \biggr\rbrace,
\end{equation}
see Kassmann~\cite{Kass11}  (for $p=2$) and Di Castro--Kuusi--Palatucci~\cite{DCKP16}.
For measurable functions $u, v: \R^n \to \eR$ 
we define the quantity
\begin{equation*}
\mathcal{E}(u,v)=\int_{\R^n} \int_{\R^n} |u(x)-u(y)|^{p-2} (u(x)-u(y))(v(x)-v(y)) k(x, y) \,dy\,dx,
\end{equation*}
provided that it is finite. Note that $\mathcal{E}(u, v)$ is well defined for 
$u \in \Wsp_{\mathrm{loc}}(\Om) \cap L^{p-1}_{sp}(\R^n)$ and $v \in C_c^\infty(\Om)$,
and that it is linear in the second argument.
Here and later, $C_c^\infty(\Om)$ denotes
the space of $C^\infty$ functions with compact support in $\Om$.
Similarly, $\Lipc(\Om)$ denotes the space of Lipschitz 
functions with compact support in $\Om$.
The \emph{support} of a function $u$ is 
$\supp u =\overline{\{x:u(x)\ne0\}}$.

We use the following definition of solutions and supersolutions of \eqref{eq-Lu=0}, 
which is now standard, see \cite{KLL23,KKL19,KKP16,KKP17,Pal18}.

\begin{deff}   \label{def-supersol}
Let $u \in W^{s, p}_{\mathrm{loc}}(\Om) \cap L^{p-1}_{sp}(\R^n)$.
Then $u$ is a 
(weak) \emph{solution} (resp.\ \emph{supersolution}) of $\LL u=0$ in $\Om$ if 
\begin{equation*}
\mathcal{E}(u, \varphi)  \ge 0 
\end{equation*}
for all  $\varphi \in C_c^{\infty}(\Omega)$
(resp.\ for all $0 \le \varphi \in C_c^{\infty}(\Omega)$).
If $u \in C(\Om)$ is a solution in $\Om$, then $u$ is \emph{$\LL$-harmonic} in $\Om$.
\end{deff}

For simplicity we will just say that $u$ is a (super)solution in $\Om$, but we always
mean with respect to $\LL u=0$ in the weak sense of Definition~\ref{def-supersol}.
If $u$ is a solution, then there is an $\LL$-harmonic function
$v$ such that $v=u$ a.e.,
see e.g.\ Di Castro--Kuusi--Palatucci~\cite[Theorem~1.4]{DCKP16}.
Their definition of solution is slightly different from ours, 
but the same proof shows that 
every solution in our sense has a representative which is
locally H\"older continuous in $\Omega$.
The  supersolutions defined in  Korvenp\"a\"a--Kuusi--Palatucci~\cite{KKP17}
are only required to satisfy 
$u \in W^{s, p}_{\mathrm{loc}}(\Om) $ and 
$u_\limminus \in L^{p-1}_{sp}(\R^n)$,
but  \cite[Lemma~1]{KKP17} shows that their
definition is equivalent to the definition above.
As usual, we let $u_\limplus=\max\{u,0\}$ and $u_\limminus=\max\{-u,0\}$.

It is often
convenient to deal with a larger class of test functions than $C_c^\infty(\Om)$.
For this purpose, we will use the spaces 
\begin{equation}     \label{eq-def-Vsp}
\Vsp(\Omega) := \biggl\lbrace u: \R^n \to \eR : 
u|_{\Omega} \in L^p(\Omega) \text{ and }  
\frac{|u(x)-u(y)|}{|x-y|^{n/p+s}} \in L^p(\Omega \times \R^n) \biggr\rbrace,
\end{equation}
equipped with the norm
\begin{align*}
\|u\|_{V^{s, p}(\Omega)} 
&:= \bigl( \|u\|_{L^p(\Omega)}^p + [u]_{V^{s, p}(\Omega)}^p \bigr)^{1/p} \\
&:= \biggl( \int_{\Omega} |u(x)|^p \,dx + \int_{\Omega} \int_{\R^n} \frac{|u(x)-u(y)|^p}{|x-y|^{n+sp}} \,dy \,dx \biggr)^{1/p},
\end{align*}
and
\begin{equation*}
V^{s, p}_0(\Omega) := \overline{C_c^{\infty}(\Omega)}^{V^{s, p}(\Omega)}.
\end{equation*}

Note that $\Vsp(\R^n)=\Wsp(\R^n)$ and that 
\[
\Lipc(\Om) \subset \Vspo(\Om) \subset
\Vsp(\Om) \subset
\Wsp(\Om) \cap L^{p-1}_{sp}(\R^n).
\]

The space $V^{s,2}(\Om)$ was (with $p=2$) introduced by
Servadei--Valdinoci~\cite{SV14} and independently by
Felsinger--Kassmann--Voigt~\cite{FKV15}.
The space $\VspoOm$ was also used by
Fiscella--Servadei--Valdinoci~\cite{FSV15},
Korvenp\"a\"a--Kuusi--Palatucci~\cite{KKP16,KKP17},
Lindgren--Lindqvist~\cite{LL17},
Kim--Lee--Lee~\cite{KLL23,KLL} and Kim--Lee~\cite{KL}.
See \cite[Remark~2.2]{KLL23} for further remarks 
on the spaces $\Vsp_0(\Om)$ and $\Vsp(\Om)$,
which were denoted by $\Wsp_0(\Om)$ and $\Vsp(\Om|\Rn)$
in~\cite{KLL23}.

In this paper, there are many places where the functions need to be 
defined pointwise everywhere in a given set, and not just a.e. 
For convenience, we will therefore assume that all functions are defined
pointwise everywhere.
When saying that $u \in \VspoOm$ we will always assume that $u \equiv 0$ 
outside $\Om$. 
To handle complications with $\infty - \infty$ 
when letting $w=u-v$ 
we will say that $w(x)=0$ whenever 
$u(x)=v(x)$ even when
$u(x)= v(x)=\pm \infty$.

\begin{prop} \label{prop-soln-Wsp}
Let $u \in W^{s, p}_{\mathrm{loc}}(\Om) \cap L^{p-1}_{sp}(\R^n)$.
Then $u$ is a  solution\/ \textup(resp.\ supersolution\/\textup) in $\Om$ if 
and only if
\begin{equation*}
\mathcal{E}(u, \varphi)  \ge 0 
\end{equation*}
for all  $\varphi \in \Vsp(\Om)$
\textup(resp.\ for all $0 \le \varphi \in \Vsp(\Om)$\textup)
with $\supp \phi \Subset \Om$.
\end{prop}

\begin{proof}
By mollification and the fact that  $\supp \phi \Subset \Om$, there are 
$\phi_j\in C_c^\infty(\Om)$ such that $\phi_j\to\phi$ in $\Vsp(\Om)$
as $j \to \infty$.
Moreover, $\phi_j \ge 0$ if $\phi \ge 0$.
The statement then follows from formula (2.9) in Kim--Lee~\cite{KL}, 
which shows that 
\[
|\mathcal{E}(u, \phi_j-\phi)| \le C(u) \|\phi_j-\phi\|_{\Vsp(\Om)}
\]
with $C(u)$ depending on $u$.
\end{proof}

\begin{cor} \label{cor-soln=sub+supersoln}
A function $u$ is a solution in $\Om$
if and only if both $u$ and $-u$ are supersolutions in $\Om$. 
\end{cor}

\begin{proof}
One implication is trivial, and the other follows
from Proposition~\ref{prop-soln-Wsp} upon using that  $\phi=\phi_\limplus-\phi_\limminus$.
\end{proof}

The following result shows that when solutions have some additional regularity, 
more test functions are allowed.

\begin{prop} \label{prop-soln-Vsp}
Let $u \in \Vsp(\Om)$.
Then $u$ is a 
solution\/ \textup(resp.\ supersolution\/\textup) in $\Om$ if 
and only if
\begin{equation*}
\mathcal{E}(u, \varphi)  \ge 0 
\end{equation*}
for all  $\varphi \in \Vspo(\Omega)$
\textup(resp.\ for all $0 \le \varphi \in \Vspo(\Omega)$\textup).
\end{prop}

In contrast to Proposition~\ref{prop-soln-Wsp} we will also
need the following lemma about convergence of truncations in order to prove
Proposition~\ref{prop-soln-Vsp}.

\begin{lem} \label{lem-truncation}
Assume that $g\in V^{s, p}(\Om)$ and $u_j \to u$ in ${V^{s, p}(\Om)}$ as $j\to\infty$.
Let 
\[
\ub=\min\{u,g\} \quad \text{and} \quad \ub_j=\min\{u_j,g\}.
\]
Then $\ub_j \to \ub$ in $V^{s, p}(\Om)$ as $j\to\infty$.
\end{lem}

This probably belongs to folklore but since we have not been able to find 
a reference, we provide a proof.
See however Costea~\cite[Lemma~2.2]{costea07}, 
which is similar but deals with a 
different nonlocal norm.
To indicate a certain subtlety in the argument, we point out that 
even though  
$|{\min\{a,m\} - \min\{b,m\}}| \le |a-b|$ 
always holds, here we need to estimate
\[
|(\min\{a,m\} - \min\{b,m\}) - (\min\{c,t\} - \min\{d,t\})|,
\]
which can in general be larger than $|(a-b)-(c-d)|$, e.g.\ when 
$a=d=m=t=0$ and $b=-c\ne0$.
So the proof is not straightforward even if $m=t=0$ (or $g \equiv 0$).
We will need this result with nonconstant $g$ when proving
Lemma~\ref{lem-Vsp0-simple}.

\begin{proof}
Let 
\begin{alignat*}{2}
v_j(x,y) & = u_j(x)-u_j(y), &\quad  v(x,y) & = u(x)-u(y), \\
\vb_j(x,y) & = \ub_j(x)-\ub_j(y), &\quad   \vb(x,y) & = \ub(x)-\ub(y), 
\end{alignat*}
and define the product measure $\mu$ on $\Om\times\R^n$ 
by
\[
d\mu(x,y) = \frac{dy \,dx}{|x-y|^{n+sp}}.
\]
Since $u_j \to u$ in ${V^{s, p}(\Om)}$, we see that
\[
\|\ub_j-\ub\|_{L^p(\Om)} \le\|u_j-u\|_{L^p(\Om)} \to0
\quad \text{and} \quad 
\|v_j-v\|_{L^p(\Om\times\R^n,\mu)} \to0,
\]
as $j \to \infty$.
It is therefore enough to show that $\|\vb_j-\vb\|_{L^p(\Om\times\R^n,\mu)} \to0$.
Assume for a contradiction that 
there is $\eps>0$ and a subsequence (also denoted by $\{u_j\}_{j=1}^\infty$)
such that
\begin{equation} \label{eq-trunc-contradiction}
\|\vb_j-\vb\|_{L^p(\Om\times\R^n,\mu)} \ge \eps
\quad \text{for } j=1,2,\dots.
\end{equation}
By passing to another subsequence if necessary we can also require that
\[
\|v_j-v\|_{L^p(\Om\times\R^n,\mu)} < 2^{-j}
\qquad \text{and} \qquad 
v_j\to v \quad  \text{$\mu$-a.e.\ in } \Om\times \R^n
\text{ as $j \to \infty$}.
\]
Clearly, also 
\begin{equation}   \label{eq-ae-conv-g} 
\vb_j\to \vb \quad \text{$\mu$-a.e.\ in $\Om\times \R^n$, as $j \to \infty$.}
\end{equation}
To use dominated convergence, we define 
\begin{equation}  \label{eq-def-v_0-g}
v_0 = \sum_{j=1}^\infty |v_j-v| \in L^p(\Om\times\R^n,\mu).
\end{equation}
The triangle inequality and a simple comparison show that
\begin{align*}
|\vb(x,y)|  &\le |{\min\{u(x),g(x)\} - \min\{u(y),g(x)\}}|  \\
& \quad   
 + |{\min\{u(y),g(x)\} - \min\{u(y),g(y)\}}| \nonumber \\
& \le |u(x) - u(y)| + |g(x) - g(y)| \\
&=: |v(x,y)| +  g_0(x,y) \nonumber
\end{align*}
and similarly
$|\vb_j(x,y)| \le |v_j(x,y)| + g_0(x,y)$.
Using~\eqref{eq-def-v_0-g} and 
$u,g \in \Vsp(\Om)$, we then get that
\begin{align*}
|\vb_j - \vb| &\le |\vb_j| + |\vb| \le |v_j| +|v| + 2g_0 \\
&\le  |v_j - v| + 2|v| + 2g_0
\le v_0 + 2|v| + 2g_0 \in L^p(\Om\times\R^n,\mu).
\end{align*}

Therefore dominated convergence and \eqref{eq-ae-conv-g} give
\[
\|\vb_j-\vb\|_{L^p(\Om\times\R^n,\mu)} \to0,
\quad \text{as } j\to\infty,
\]
contradicting \eqref{eq-trunc-contradiction}.
\end{proof}

\begin{proof}[Proof of Proposition~\ref{prop-soln-Vsp}]
Suppose that $u$ is a supersolution in $\Om$ and let $0 \leq \varphi \in \Vspo(\Om)$.
By definition,  there exist $\phi_j \in C^\infty_c(\Om)$
such that $\phi_j \to \varphi$ in $\Vsp(\Om)$ as $j \to \infty$.
Hence
$(\phi_j)_\limplus \in \Lip_c(\Om) \subset \Vsp(\Om)$.
It now follows from H\"older's inequality  and Lemma~\ref{lem-truncation} that
\begin{equation*}
|\mathcal{E}(u, (\phi_j)_\limplus) - \mathcal{E}(u, \varphi)|
\leq 2\Lambda [u]_{\VspOm}^{p-1}  [(\phi_j)_\limplus-\varphi]_{\VspOm}
\to 0,  \quad \text{as } j \to \infty,
\end{equation*}
where $\La\ge1$ is as in~\eqref{eq-comp-(x,y)}.
Since $\supp {(\phi_j)_\limplus} \Subset \Om$, 
Proposition~\ref{prop-soln-Wsp} shows that
\[
   \mathcal{E}(u, \varphi) 
  = \lim_{j \to \infty} \mathcal{E}(u, (\phi_j)_\limplus) \geq 0.
\]
The proof for solutions is similar or follows from the above 
and Corollary~\ref{cor-soln=sub+supersoln}.
\end{proof}

The following two lemmas may also belong to folklore.
They may actually have been used in some earlier papers,
but we are not aware of any written proofs.
It seems that their proofs require some form
of Lemma~\ref{lem-truncation}.

\begin{lem} \label{lem-Vspo-limplus}
If $u \in \Vspo(\Om)$, then $u_\limplus \in \Vspo(\Om)$.
Moreover, there are $0  \le v_j\in C_c^\infty(\Om)$ 
such that $v_j \to u_\limplus$ in $\Vsp(\Om)$ as $j \to \infty$.
\end{lem}

\begin{lem} \label{lem-Vsp-C}
Assume that $\Om$ is bounded.
If $u \in \Vsp(\Om)\cap C(\R^n)$ and $u=0$ on $\Omc$, then $u \in \Vspo(\Om)$.
\end{lem}

Lemma~\ref{lem-Vsp-C} is not true in general without the assumption $u \in C(\R^n)$, see e.g.\ Fiscella--Servadei--Valdinoci~\cite[Remark~7]{FSV15}.

\begin{proof}[Proof of Lemma~\ref{lem-Vspo-limplus}]
By assumption, there are $u_j \in C^\infty_c(\Om)$
such that $u_j \to u$ in $\Vsp(\Om)$ as $j \to \infty$.
Thus
$(u_j)_\limplus \in \Lip_c(\Om) \subset \Vsp(\Om)$.
By Lemma~\ref{lem-truncation}, $(u_j)_\limplus \to u_\limplus$ in $\Vsp(\Om)$ as $j \to \infty$.
Since $\supp {(u_j)_\limplus} \Subset \Om$, 
there are also mollifications $v_j\in C_c^\infty(\Om)$ of $(u_j)_\limplus$
such that $0 \le v_j \to u_\limplus$ in $\Vsp(\Om)$,
i.e.\ $u_\limplus \in \Vspo(\Om)$.
\end{proof}

\begin{proof}[Proof of Lemma~\ref{lem-Vsp-C}]
By splitting $u$ into $u_\limplus$ and $u_\limminus$,
we may assume that $u \ge 0$.
As $\Om$ is bounded, $u-1/j \to u$ in $\Vsp(\Om)$, and
thus by Lemma~\ref{lem-truncation} also $(u-1/j)_\limplus \to u$ in $\Vsp(\Om)$,
as $j \to \infty$.
Since $\supp {(u-1/j)_\limplus} \Subset \Om$ (due to $u\in C(\Rn)$),
there are also mollifications $v_j\in C_c^\infty(\Om)$ of $(u-1/j)_\limplus$
such that $0 \le v_j \to u$ in $\Vsp(\Om)$,
i.e.\ $u \in \Vspo(\Om)$.
\end{proof}

\begin{lem}\label{lem-Vsp0-simple}
Let $\Om_1, \Om_2 \subset \R^n$ be open sets and
$0\le u\in \Vsp(\Om_1\cap \Om_2)$.
Assume that  $u \le u_j$ a.e.\ in $\R^n$ for some
$u_j \in \Vsp_0(\Om_j)$,  $j=1,2$.
Then $u \in \Vsp_0(\Om_1\cap \Om_2)$.
\end{lem}

\begin{proof}
For $j=1,2$, let $v_{i,j}\in C_c^\infty(\Om_j)$ be such that $v_{i,j}\to u_j$
in $\Vsp(\Om_j)$ as $i \to \infty$.
Then by Lemma~\ref{lem-truncation},
\[
v_i:= \min\{u,(v_{i,1})_\limplus,(v_{i,2})_\limplus\} \to u
\quad \text{in }  \Vsp(\Om_1 \cap \Om_2), \text{ as } i \to \infty.
\]
Since $\supp v_i \Subset\Om_1 \cap \Om_2$, a standard mollification argument finishes the proof.
\end{proof}

\begin{cor}\label{cor-police}
Let $u \in \VspOm$ and $v_1,v_2 \in \Vspo(\Om)$.
If $v_1 \leq u \leq v_2$ a.e.\ in $\R^n$, then $u \in \Vspo(\Om)$.
\end{cor}

\begin{proof}
By taking differences, we may assume that $v_1 \equiv 0$.
The result now follows from Lemma~\ref{lem-Vsp0-simple} upon letting $u_j=v_2$
and $\Om_j = \Om$, $j=1,2$.
\end{proof}

\section{The Dirichlet and obstacle problems}
\label{sec-D-obstacle}

\emph{From now on, we always assume that $\Om \subset \Rn$
is a nonempty bounded  open set.}

\medskip

In this section, we study the obstacle problem, which plays an important role 
in the development of potential theory.
The Dirichlet problem (in the Sobolev sense)
is a special case of the obstacle problem, namely with 
the obstacle $\psi \equiv -\infty$, i.e.\ without an obstacle.
Recall from Section~\ref{sec-sobolev} that all functions are defined
pointwise everywhere
and that $v-g\in \Vsp_0(\Om)$ implies
that $v \equiv g$ in $\Omc$.

\begin{deff}
For $g\in \Vsp(\Om)$ and $\psi:\Om\to\eR$, let
\[
\K_{\psi,g}(\Om) = \{v \in \Vsp(\Om):  v-g\in \Vsp_0(\Om) \text{ and }
v\ge \psi \text { a.e.\ in } \Om\}.
\]
A function $u\in \K_{\psi,g}(\Om)$  
is a \emph{solution of the $\K_{\psi,g}(\Om)$-obstacle problem} if 
\[
\E(u,v-u) \ge 0 \quad \text{for all $v\in \K_{\psi,g}(\Om)$.}
\]
\end{deff}

It follows directly from the definition that a solution of the 
$\K_{\psi,g}(\Om)$-obstacle
problem is a supersolution of $\LL u = 0$ in $\Om$, 
and moreover that it is a solution of $\LL u = 0$ if   $\psi \equiv -\infty$,
i.e.\ for the Dirichlet problem.

The obstacle problem was studied in detail by 
Korvenp\"a\"a--Kuusi--Palatucci~\cite{KKP16}.
However, here we follow Kim--Lee~\cite{KL} and use a somewhat
different definition of the obstacle problem that is more suitable for
our purposes.

In this paper it will also be important to consider
lsc-regularizations which we define in the following way.
For any measurable function $u:\R^n \to \eR$ we define its
\emph{lsc-regularization} in $\Om$ as
\begin{equation}   \label{eq-def-uhat}
\uhat(x) :=   \begin{cases}
   u(x), & \text{if } x \in \Omc, \\
\displaystyle\essliminf_{y\to x} u(y), & \text{if } x \in \Om.
\end{cases}
\end{equation}
We also say that $u$ is \emph{lsc-regularized} (in $\Om$) if $u \equiv \uhat$.

\begin{thm} \label{thm-obst-KL-solv}
\textup{(Kim--Lee~\cite[Theorem~4.9]{KL} and 
Korvenp\"a\"a--Kuusi--Palatucci~\cite[Theorem~9]{KKP17})}
Let $g\in \Vsp(\Om)$ and $\psi:\Om\to\eR$.
If  $\K_{\psi,g}(\Om) \ne \emptyset$,
then the $\K_{\psi,g}(\Om)$-obstacle problem is solvable,
and the solution $u$ is unique up to measure zero. 

Moreover, $u$ is a supersolution, $\uhat=u$ a.e., and 
$\uhat$ is the unique lsc-regularized solution of the
$\K_{\psi,g}(\Om)$-obstacle problem.
\end{thm}

Note that the lsc-regularization is only within $\Om$.

\begin{proof}
The existence and a.e.-uniqueness, and that $u$ is a supersolution was
shown by Kim--Lee~\cite[Theorem~4.9]{KL}.
Since $u$ is a supersolution, it follows from \cite[Theorem~9]{KKP17}
that $\uhat=u$ a.e.
Thus $\uhat$ is  also a solution
of the $\K_{\psi,g}(\Om)$-obstacle problem and its 
pointwise uniqueness follows immediately from the 
a.e.-uniqueness of $u$.
\end{proof}

For functions $g \in \VspOm$ we can solve
the Dirichlet problem in the Sobolev sense. 
To make it precise we define $Hg$ as follows.

\begin{thm} 
Let 
$g \in \VspOm$.
Then there is a unique function $Hg=H_\Om g:\R^n \to \eR$ with the following properties\/\textup{:}
\begin{enumerate}
\item
  $Hg$ is $\LL$-harmonic in $\Om$,
\item
  $Hg-g \in \VspoOm$, and in particular
  $Hg \equiv g$ outside $\Om$.
\end{enumerate}
\end{thm}

\begin{proof}
This essentially follows from Theorem~\ref{thm-obst-KL-solv}: 
There is a solution $u$ of the $\K_{-\infty,g}(\Om)$-obstacle problem which is unique
up to measure zero.
At the same time,
$u$ is a solution of $\LL u =0$ in $\Om$ and thus, after redefinition
on a set of measure zero within $\Om$, can be required to be $\LL$-harmonic in $\Om$.
\end{proof}

We will need several comparison principles in this paper.
The following result is  for obstacle problems.

\begin{thm}\label{thm-comp-obs} 
\textup{(Comparison principle for obstacle problems)}
Let $u_i$ be a solution of the $\K_{\psi_i, g_i}(\Om)$-obstacle problem for $i=1,2$. 
If $\psi_1 \le \psi_2$ a.e.\ in $\Om$ and 
$(g_1-g_2)_\limplus \in \Vsp_0(\Om)$,
then $u_1 \le u_2$ a.e.\ in $\Om$.
\end{thm}

Theorem~\ref{thm-comp-obs} is a direct consequence of the following lemma
and the fact that solutions of the obstacle problem are 
supersolutions of $\LL u=0$.

\begin{lem}   \label{lem-comp-first}
Let $u$ be a solution of the $\K_{\psi,g}(\Om)$-obstacle problem. 
If $v \in \Vsp(\Om)$ is a supersolution in $\Om$ such that 
$\min\{u,v\} \in \K_{\psi,g}(\Om)$, then $v \geq u$ a.e.\ in $\Om$.
\end{lem}

\begin{proof}
Since $v$ is a supersolution and
$w:=u-\min\{u,v\} \in \Vsp_0(\Om)$ is nonnegative, it follows from 
Proposition~\ref{prop-soln-Vsp} that $\E(v,w)\ge0$.
Moreover, $u-w=\min\{u,v\} \in \K_{\psi,g}(\Om)$
and thus $-\E(u,w)=\E(u,(u-w)-u) \ge 0$, by
the definition
of the $\K_{\psi,g}(\Om)$-obstacle problem.
Hence
\begin{align*}
0 &\leq \mathcal{E}(v, w) - \mathcal{E}(u, w) \\
&= \int_{\{u>v\}} \int_{\{u>v\}} (\Phi(v(x)-v(y))-\Phi(u(x)-u(y)))(w(x)-w(y)) k(x, y) \,dy\,dx \\
&\quad+ \int_{\{u>v\}} \int_{\{u\leq v\}} (\Phi(v(x)-v(y))-\Phi(u(x)-u(y)))w(x) k(x, y) \,dy\,dx \\
&\quad+ \int_{\{u\leq v\}} \int_{\{u>v\}} (\Phi(v(x)-v(y))-\Phi(u(x)-u(y)))(-w(y)) k(x, y) \,dy\,dx \\
&=: I_1+I_2+I_3,
\end{align*}
where $\Phi(t):=|t|^{p-2}t$ is a strictly increasing function. 
If $u(x)>v(x)$ and $u(y) \leq v(y)$, then $v(x)-v(y) < u(x)-u(y)$ 
and hence 
\begin{equation}  \label{eq-nonpos-I2}
\Phi(v(x)-v(y)) < \Phi(u(x)-u(y)). 
\end{equation}
This leads to $I_2 \leq 0$, and similarly we have $I_3 \leq 0$. 
On the other hand, if $u(x)>v(x)$ and $u(y)>v(y)$, then
\begin{equation*}
(\Phi(v(x)-v(y))-\Phi(u(x)-u(y)))(w(x)-w(y)) \le 0
\end{equation*}
since $w(x)-w(y) = (u(x)-u(y)) - (v(x)-v(y))$ and
\begin{equation*}
(\Phi(a)-\Phi(b))(b-a) = (b-a)\int_b^a (p-1)|t|^{p-2} \,dt \le 0
\end{equation*}
whenever $a, b \in \R$.
Thus, we obtain $I_1 \le 0$, which leads us to $I_1=I_2=I_3=0$. 
In particular, since the integrand in $I_2$ is strictly negative
(by~\eqref{eq-nonpos-I2}), we conclude that
either $|\{x: u(x) > v(x) \}|=0$ (which gives the claim) or 
\[
0=|\{y: u(y) \le v(y) \}| = |\{y: w(y) =0 \}| \ge |\Om^c| > 0,
\]
which is impossible.
(Here $|\cdot|$ denotes the Lebesgue measure.)
\end{proof}

\begin{cor} \label{cor-comparison-sol}
\textup{(Comparison principle for $H$-solutions)}
Let $g_1, g_2 \in \Vsp(\Om)$ be such that 
$(g_1-g_2)_\limplus \in \Vsp_0(\Om)$.
Then $Hg_1 \le Hg_2$ in $\R^n$.
\end{cor}

\begin{proof}
By Theorem~\ref{thm-comp-obs} (with $\psi_1=\psi_2\equiv-\infty$), 
$Hg_1 \le H g_2$ a.e.\ in $\Om$.
Since $Hg_1$ and $H g_2$ are continuous  in $\Om$, 
the inequality holds everywhere in $\Om$.
Moreover, $Hg_1 =g_1 \le g_2 = Hg_2$ in $\Omc$.
\end{proof}

\section{Sobolev regularity}\label{sec-Sob-reg}

We follow Kim--Lee--Lee~\cite[Theorem~1.1]{KLL23} in defining
Sobolev regularity as follows.
(They called it regularity, but we want to distinguish it from Perron regularity,
which a priori could be a different property even though it will be shown to be equivalent
in Theorem~\ref{thm-main-reg-Perron}.)

\begin{deff}
  A boundary point $x_0 \in \bdy \Om$ is \emph{Sobolev regular} (with respect to $\LL$
and $\Om$)
  if 
  \[
    \lim_{\Om \ni x \to x_0} H g(x)=g(x_0)
    \quad \text{for every } g \in \VspOm \cap C(\R^n).
  \]  
  Otherwise we say that $x_0$ is \emph{Sobolev irregular}.
\end{deff}  

The main result in Kim--Lee--Lee~\cite{KLL23} is the Wiener
criterion in the following form. 
(The case $sp>n$
 was treated incorrectly in \cite{KLL23}. 
A correction was made later in Kim--Lee--Lee~\cite{KLL}, see the last paragraph of Remark~1.5 in \cite{KLL}.)

\begin{thm} \label{thm-Wiener}
\textup{(Wiener criterion \cite[Theorem~1.1]{KLL23})}
A boundary point $x_0 \in \bdy \Om$ is Sobolev regular if and only if 
\begin{equation} \label{eq-Wiener}
  \int_0^1 \biggl(
  \frac{\csp(\itoverline{B(x_0, \rho)}\setm \Om,B(x_0,2\rho))}{\rho^{n-sp}}
          \biggr)^{1/(p-1)} \, \frac{d\rho}{\rho} = \infty,
\end{equation}
where the condenser capacity $\csp$ is defined in Definition~\ref{deff-cpt} below.
\end{thm}

Note that since \eqref{eq-Wiener} is independent of the kernel $k$, so 
is the Sobolev regularity of $x_0$.
It is easy to see that if $\Om$ satisfies the exterior cone condition or
the exterior corkscrew condition at $x_0$, then the 
Wiener integral~\eqref{eq-Wiener} diverges, and hence
$x_0$ is Sobolev regular.
Also the following important consequences follow immediately.

\begin{cor} \label{cor-local-subset}
Let $G$ be a bounded open set and let $x_0 \in \bdy G \cap \bdy \Om$.
\begin{enumerate}
\item
\textup{(Sobolev regularity is a local property)}
If $B \cap G = B \cap \Om$ for some  ball  $B \ni x_0$,
then $x_0$ is Sobolev regular for $G$ if and only it is Sobolev regular for $\Om$.
\item
If $G \subset \Om$ 
and  $x_0$ is Sobolev regular for $\Om$, then it is also Sobolev regular for~$G$.
\end{enumerate}
\end{cor}

Before characterizing Sobolev regularity, we need to define the $\LL$-potential.

\begin{deff}  \label{def-pot}
Let $K \Subset \Om$ 
be compact and  
$\psi \in C^{\infty}_{c}(\Omega)$ be such that $\psi = 1$ on $K$.
Then the function $H_{\Om \setm K}\psi$ is called the \emph{$\LL$-potential}
 of $K$ in $\Omega$ and is denoted by $\mathfrak{R}(K, \Omega)$.
\end{deff}

The $\LL$-potential is independent of the choice of $\psi$
by Corollary~\ref{cor-comparison-sol},
and it
is a supersolution in $\Om$ by Lemma~2.16(iv) in Kim--Lee--Lee~\cite{KLL23}.
It therefore follows from 
Theorem~9 in
Korvenp\"a\"a--Kuusi--Palatucci~\cite{KKP17}
that its lsc-regularization
$\pothat(K, \Omega)$ within $\Om$ satisfies 
$\pothat(K, \Omega)=\pot(K, \Omega)$ a.e.,
and is thus also a supersolution in $\Om$. 
(Moreover, $\pothat(K, \Omega)$ is $\LL$-superharmonic in $\Omega$,
by Theorem~\ref{thm:KKP17}, but we will not use this fact.)

We are now ready to give our first characterizations of Sobolev regularity. 
For a ball $B=B(x_0,r):=\{x \in \Rn: |x-x_0| <r\}$, we let
$\la B=B(x_0, \la r)$.

\begin{thm}\label{thm-Sobolev-reg}
  Let $x_0 \in \bdy \Om$.
Then the following are equivalent\/\textup{:}
\begin{enumerate}
\item \label{a-unbdd}
  $x_0$ is Sobolev regular.
\item \label{a-reg}
\begin{equation*}
\lim_{\Om \ni x \to x_0} Hg(x) = g(x_0) \quad\text{for every bounded }
    g \in \VspOm \cap C(\Rn).
\end{equation*}
\item \label{a-contx0}
  \[
    \lim_{\Om \ni x \to x_0} Hg(x)=g(x_0)
  \]  
for every bounded $ g \in \VspOm$ 
that is continuous at $x_0$.
\item \label{a-d}
  \[
    \lim_{\Om \ni x \to x_0} Hd_{x_0}(x)=0, 
\quad \text{where }   d_{x_0}(x):=\min\{1,|x-x_0|\}.
  \]  
\item \label{a-potential}
\begin{equation*}
    \pothat(\clB \setminus \Omega, 2B)(x_0)=1
\quad \text{for every open ball } B \ni x_0.
\end{equation*}
\item \label{a-potential-seq}
There is a decreasing sequence of radii $r_j \searrow 0$ such that
\begin{equation*}
    \pothat(\itoverline{B(x_0,r_j)} \setminus \Omega, B(x_0,2r_j))(x_0)=1
\quad \text{for each } j.
\end{equation*}
\end{enumerate}  
\end{thm}

Note that 
\begin{equation} \label{eq-pot}
\pothat(\clB \setminus \Omega, 2B)(x_0)
=\liminf_{\Om\ni x\to x_0} \pot(\clB \setminus \Omega, 2B) (x),
\end{equation}
since the potential $\pot(\clB \setminus \Omega, 2B)$ is continuous in
the open set $2B \setm (\clB \setm \Om)$.

\begin{proof}[Proof of Theorem~\ref{thm-Sobolev-reg}]
\ref{a-unbdd}\imp\ref{a-reg}\imp\ref{a-d} and \ref{a-contx0}\imp\ref{a-reg}
Since $d_{x_0} \in \VspOm \cap C(\Rn)$, 
these implications are  trivial.

$\neg$\ref{a-unbdd}\imp$\neg$\ref{a-reg}
By Theorem~\ref{thm-Wiener}, the Wiener integral in~\eqref{eq-Wiener} converges.
The function $u_\rho$, with $\rho>0$ sufficiently small, constructed
in the proof of the necessity part of the Wiener criterion
in~\cite[end of Section~5]{KLL23}
is bounded. Hence this implication holds.

\ref{a-d}\imp\ref{a-contx0}
By adding a constant if necessary, we
may assume that $g(x_0)=0$. Let $\eps >0$.
Then there is $m>0$ such that 
$g \le md_{x_0}+\eps$ in $\R^n$ and thus
$(g-md_{x_0}-\eps)_\limplus \in V^{s, p}_0(\Om)$.
By the comparison principle (Corollary~\ref{cor-comparison-sol}), we see that
\[
  \limsup_{\Om \ni x \to x_0} Hg(x)
  \le     \lim_{\Om \ni x \to x_0} m H d_{x_0}(x)+ \eps=\eps.
\]
Letting $\eps \to 0$ shows that 
\[
  \limsup_{\Om \ni x \to x_0} Hg(x) \le 0.
\]
Applying this also to $-g$ shows that 
\[
  \lim_{\Om \ni x \to x_0} Hg(x) = 0.
\]

\ref{a-unbdd}\imp\ref{a-potential}
Let $B \ni x_0$ be an  open ball and 
$\psi \in C^{\infty}_{c}(2B)$ be such that $\psi = 1$ on 
$\clB \setm \Om$.
By Corollary~\ref{cor-local-subset}, $x_0$ is Sobolev regular also 
with respect to $G:=2B \setm (\clB \setm \Om)$.
Thus, by \eqref{eq-pot},
\[
   \pothat(\clB \setm \Om, 2B)(x_0)
=\liminf_{\Om\ni x\to x_0} \pot(\clB \setm \Om, 2B) (x)
=\liminf_{G\ni x\to x_0} H_G \psi (x)
=\psi(x_0)=1.
\]

\ref{a-potential}\imp\ref{a-potential-seq}
This is trivial.

\ref{a-potential-seq}\imp\ref{a-d}
Fix $j$ such that 
$0<r_j<1$. 
Let  $B=B(x_0,r_j)$
and $u =\pot(\clB\setm\Om,2B)$.
Since 
\[
d_{x_0} \le2r_j \le 1+2r_j-u \text{ in } 2B
\quad \text{and} \quad
d_{x_0}  \le 1\le 1+2r_j-u \text{ in } \R^n\setm 2B,
\] 
we have by the comparison principle (Corollary~\ref{cor-comparison-sol}) that
\begin{equation}   \label{eq-est-dx0-with-pot}
H_\Om d_{x_0} \le 1+2r_j -H_\Om u.
\end{equation}

We shall use the comparison principle (Corollary~\ref{cor-comparison-sol}) 
once more to show that
$H_\Om u \ge u$
in $\Om\cap 2B$, which will then conclude the proof.
Note that $H_\Om u$ and $u$ are solutions  in $\Om\cap 2B$ with boundary data
$H_\Om u$ and $u$, respectively.

To see that
$H_\Om u \ge u$ in $\Om\cap 2B$, 
it therefore suffices to show that
$(u-H_\Om u)_\limplus \in \Vsp_0(\Om\cap 2B)$.
This  follows from  Lemma~\ref{lem-Vsp0-simple}, since
\[
0\le (u-H_\Om u)_\limplus \le u_\limplus \in \Vsp_0(2B)
\quad \text{and} \quad
(u-H_\Om u)_\limplus \in \Vsp_0(\Om),
\]
by Lemma~\ref{lem-Vspo-limplus}.
Hence, by  \eqref{eq-est-dx0-with-pot}, \eqref{eq-pot} and~\ref{a-potential-seq},
\[
0 \le \limsup_{\Om\ni x\to x_0}  H_\Om d_{x_0} (x)
\le 1+ 2r_j -\liminf_{\Om\ni x\to x_0} u(x)  
= 2r_j.
\]
Letting $j\to\infty$ shows that \ref{a-d} holds.
\end{proof}

\section{Capacities}\label{sec-capacity}

In order to prove the Kellogg property (Theorem~\ref{thm-kellogg})
we first need to study capacities.
Two types of capacities are used in this paper:
the Sobolev capacity and
the condenser capacity.
In contrast to Kim--Lee--Lee~\cite{KLL23}
we will need capacities for noncompact sets,
e.g.\ in the Kellogg property (Theorem~\ref{thm-kellogg}).
We therefore make the following definitions.

\begin{deff} \label{deff-cpt}
The \emph{condenser capacity} of a compact set $K \Subset \Om$ is given by
\[
    \cpt(K,\Om) = \inf_u {[u]_{W^{s,p}(\R^n)}^p},
\]
where the infimum is taken over all  $u\in C_c^\infty(\Om)$
such that 
$u \ge 1$  on $K$.
\end{deff}

\begin{deff} 
The \emph{Sobolev capacity} of a compact set $K \Subset \Rn$ is defined by
\begin{equation}  \label{eq-def-norm-cap}
 \Cpt(K)   =\inf_u {\|u\|_{\Wsp(\Rn)}^p},
\end{equation}
where the infimum is taken over all $u\in C_c^\infty(\Rn)$ such that 
$u \ge 1$ on $K$.
\end{deff}

Both capacities are then extended first to open and then to arbitrary sets 
in the usual way as follows:
For open sets $G$,
\begin{equation}   \label{eq-cap-G}
\begin{aligned}
\cpt(G,\Om) & = \sup_{\substack{K \text{ compact}\\ K \Subset G}} \cpt(K,\Om), && \text{if $G \subset \Om$}, \\
\Cpt(G) &= \sup_{\substack{K \text{ compact}\\ K \Subset G}} \Cpt(K), && \text{if $G \subset \Rn$}.
\end{aligned}
\end{equation}
Similarly, for arbitrary  sets~$E$,
\begin{equation}  \label{eq-cap-E}
\begin{aligned}
\cpt(E,\Om) &= \inf_{\substack{G \text{ open}\\ E \subset G \subset \Om}} \cpt(G,\Om), && \text{if $E \subset \Om$}, \\
\Cpt(E) &= \inf_{\substack{G \text{ open}\\ E \subset G}} \Cpt(G), && \text{if $E \subset \Rn$}.
\end{aligned}
\end{equation}

It is worth noticing that~\eqref{eq-cap-G} and~\eqref{eq-cap-E} do not change the 
definition of capacity for compact sets, cf.\ Adams--Hedberg~\cite[Proposition~2.2.3]{AH}.
Indeed, let $u\in C_c^\infty(\Rn)$ be such that $u\ge1$ on a compact set~$K$.
For every $0<\eps<1$ and every compact $K'\subset G_\eps:=\{x: u(x)>1-\eps\}$, 
the function $u/(1-\eps)$ is admissible for $\Cpt(K')$.  
Since $G_\eps$ is open and $K\subset G_\eps$, we thus get 
by~\eqref{eq-def-norm-cap} and~\eqref{eq-cap-G} 
(without using \eqref{eq-cap-E} as a definition) that
\[
\Cpt(K) \le \inf_{\substack{G \text{ open}\\ K \Subset G}} \Cpt(G) \le \Cpt(G_\eps) 
= \sup_{\substack{K' \text{ compact}\\ K'\Subset G_\eps}} \Cpt(K')
\le \frac{\|u\|_{\Wsp(\Rn)}^p}{(1-\eps)^p}.
\]
Now letting $\eps\to0$ and taking the infimum over all such 
$u$ shows that \eqref{eq-cap-E} holds for $\Csp(K)$.
The argument for $\cpt(K,\Om)$ is the same.

Our definition of the condenser capacity coincides with the definitions in 
Kim--Lee--Lee~\cite[p.~1968]{KLL23} (for compact $K \Subset \Om$)
and~\cite[Definition~5.1]{KLL}.
For compact and open sets $E$ it also coincides 
with the capacity $\capp_{B^s_p}$
considered in Bj\"orn~\cite{JBWien}.
With our definition, we can now extend Lemma~1.3 in \cite{JBWien} 
to arbitrary sets in the following way.
We will use this result when proving the 
Kellogg property (Theorem~\ref{thm-kellogg}) in
Section~\ref{sec-kellogg}.

Here $A \simeq A'$ means that $A/C\le A'\le CA$ 
for some constant $C$ independent
of the quantities $A$ and $A'$.
In addition to balls in $\R^n$ we will also
use
balls in $\R^{n+1}$, which we denote by
\[
B'(z_0,r)=\{z \in \R^{n+1}: |z-z_0| <r\}.
\]

\begin{lem}\label{lem-cap-extended}
\textup{(\cite[Lemma~1.3]{JBWien})}
Let $E\subset  \itoverline{B(x_0,r)}$ and $z_0=(x_0,0)\in\R^{n+1}$.
Then 
\[
\cpt(E,B(x_0,2r)) \simeq \cpa(E\times\{0\}, B'(z_0,2r)),
\]
where $\cpa$ is the weighted \p-capacity in $\R^{n+1}$ associated with the weight
\[
  w(x,t)=|t|^a \quad \text{for }
  a=p(1-s)-1,
\]
as in  Heinonen--Kilpel\"ainen--Martio\/~\textup{\cite[Chapter~2]{HeKiMa}}.

\end{lem}

\begin{proof}
For open sets $E$,
this follows from \cite[Lemma~1.3]{JBWien}.
Let $E\subset \itoverline{B(x_0,r)}$ be arbitrary.
Then
\begin{align*}
\cpt(E,B(x_0,2r))
&= \inf_{G\supset E} \cpt(G,B(x_0,2r)) \\
&\simeq \inf_{G\supset E} \cpa(G\times \{0\},B'(z_0,2r)),
\end{align*}
where the infima are taken over all open sets $G\subset \Rn$ such that $E\subset G\subset B(x_0,2r)$.
At the same time, by the definition in~\cite[Chapter~2]{HeKiMa},
\[
\cpa(G\times \{0\},B'(z_0,2r)) = \inf_{U\supset G\times \{0\}} \cpa(U,B'(z_0,2r)),
\]
with the infimum taken over all open sets $U\subset \R^{n+1}$ such that 
$G\times \{0\}\subset U\subset B'(z_0,2r)$.
This finally shows that 
\begin{align*}
\cpt(E,B(x_0,2r))
&\simeq \inf_{U\supset E\times \{0\}} \cpa(U,B'(z_0,2r)) \\
&= \cpa(E\times\{0\},B'(z_0,2r)),
\end{align*}
where the infimum is taken over all open sets $U\subset \R^{n+1}$ such that 
$E\times \{0\}\subset U\subset B'(z_0,2r)$, 
and the definition in~\cite[Chapter~2]{HeKiMa}
was used for the last equality.
\end{proof}

The following result connects the two capacities $\cpt$ and $\Cpt$
and shows that they have the same zero sets.
For completeness and the reader's convenience we provide the proof.

\begin{prop} \label{prop-cp-Cp}
Let $E  \Subset \Om$.
Then  
\begin{equation} \label{eq-cp-Cp-Om}
 \frac{\Csp(E)}{C(1+(\diam\Om)^{sp})} \le
  \csp(E,\Om) \le C\biggl(1+\frac{1}{\dist(E,\Omc)^{p}}  \biggr) \Csp(E),
\end{equation}  
where $C$ only depends on $n$, $s$ and $p$.
In particular, $\Cpt(E)=0$ if and only if  $\cpt(E,\Om)=0$.
\end{prop}

\begin{proof}
Let $u\in C_c^\infty(\Om)$. 
Then clearly, for any $z\in\Rn$ with $\dist(z,\Om)=2\diam\Om$, 
\begin{equation*}
\int_{\Om}|u(x)|^p  \,  dx
\le C (\diam\Om)^{sp} \int_{\Om}\int_{B(z,\diam\Om)} \frac{|u(x)-u(y)|^p}{|x-y|^{n+s p}}  \, dy\, dx,
\end{equation*}
which gives the fractional Poincar\'e inequality
\begin{equation} \label{eq-PI}
\|u\|_{W^{s,p}(\R^n)}^p \le C(1+(\diam\Om)^{sp})[u]_{W^{s,p}(\R^n)}^p.
\end{equation}
See Maz{\cprime}ya--Shaposhnikova~\cite{MS} and Ponce~\cite{Ponce} 
for much more general  Poincar\'e inequalities.
Taking the infimum over all $u\in C_c^\infty(\Om)$, which are admissible 
in the definition  of $\cpt(K,\Om)$ for compact $K \Subset \Om$, 
and then using \eqref{eq-cap-G}
and \eqref{eq-cap-E} 
proves the first inequality in~\eqref{eq-cp-Cp-Om} and
the ``if'' part in the last statement of the lemma. 

For the second inequality and the ``only if'' part, let $\eta\in C_c^\infty(\Om)$ be such that $0\le\eta\le1$
everywhere and $\eta=1$ in an open neighbourhood $V$ of $E$.
If $u\in C_c^\infty(\Rn)$ then $u\eta \in C_c^\infty(\Om)$ and 
\begin{equation}   \label{eq-u-eta}
|(u\eta)(x)- (u\eta)(y)|  \le |u(x)| |\eta(x)-\eta(y)|  + |u(x)-u(y)|.
\end{equation}
Note that 
\[
\int_{\Rn}  \frac{|\eta(x)-\eta(y)|^p}{|x-y|^{n+s p}}  \, dy
\le \int_{B(x,1)}  \frac{M^p\,dy}{|x-y|^{n+(s-1) p}} +
\int_{\Rn\setm B(x,1)}  \frac{dy}{|x-y|^{n+s p}} < \infty,
\]
where $M$ is the Lipschitz constant of $\eta$.
Since $\eta$ can be chosen so that $M\le 2/{\dist(E,\Om^c)}$,  we get
from~\eqref{eq-u-eta} that
\[
[u\eta]_{W^{s,p}(\R^n)}^p 
\le C (A \|u\|^p_{L^p(\Rn)} + [u]_{W^{s,p}(\R^n)}^p) 
       \le  CA \|u\|_{W^{s,p}(\R^n)}^p,
\]
where $A=1+1/{\dist(E,\Om^c)^{p}}$ and $C$ only depends on $n$, $s$ and $p$.
As before, taking the infimum over all $u\in C_c^\infty(\Rn)$, which are admissible 
in the definition  of $\Cpt(K)$ for compact $K \Subset V$, 
and then using \eqref{eq-cap-G}
and~\eqref{eq-cap-E} concludes the proof.
\end{proof}

\begin{lem} \label{lem-sp<=n}
Let $x_0 \in \Om$. 
Then the following are equivalent\/\textup:
\begin{enumerate}
\item
$sp \le n$,
\item
$\Csp(\{x_0\})=0$, 
\item
$\csp(\{x_0\},\Om)=0$.
\end{enumerate}
\end{lem}

\begin{proof}
This follows directly from Lemma~2.17
in Kim--Lee--Lee~\cite{KLL23} together with
Proposition~\ref{prop-cp-Cp} and~\eqref{eq-cap-E}. 
\end{proof}

The following lemma shows in particular that a set
with zero Sobolev capacity has zero measure. 
It follows from Proposition~\ref{prop-cp-Cp} that the same is true for the condenser capacity.
Here and later, $|\cdot|$ denotes the Lebesgue measure and
$|\cdot|_{\Outer}$ is the Lebesgue outer measure.
For completeness and the reader's convenience we provide the proof.

\begin{lem}\label{lem-zero-cap}
If $E\subset \R^n$, then
$\Cpt(E) \ge |E|_{\Outer}$.
\end{lem}

\begin{proof}
If $K$ is compact, then
\begin{equation*} 
 \Cpt(K)   
=\inf_u {\|u\|_{\Wsp(\Rn)}^p}
\ge \inf_u {\|u\|_{L^p(\Rn)}^p}
\ge |K|,
\end{equation*}
where the infima are taken over all $u\in C_c^\infty(\Rn)$ such that 
$u \ge 1$ on $K$.
It then follows, by the regularity of the Lebesgue measure,  that for open $G$,
\begin{equation*} 
\Cpt(G) 
= \sup_{\substack{K \text{ compact}\\ K \Subset G}} \Cpt(K) 
\ge \sup_{\substack{K \text{ compact}\\ K \Subset G}} |K|
= |G|.
\end{equation*}
And  then, by the definition of the outer
Lebesgue measure,
for arbitrary sets~$E$,
\begin{equation*} 
\Cpt(E) 
= \inf_{\substack{G \text{ open}\\ E \subset G}} \Cpt(G) 
\ge  \inf_{\substack{G \text{ open}\\ E \subset G}} |G|
=|E|_{\Outer}.
\qedhere
\end{equation*}
\end{proof}

The following lemma will be needed when proving Theorem~\ref{thm-Hg=Pg}.

\begin{lem} \label{lem-Wsp-Csp}
Let $E \subset\Rn$.
Then
\begin{equation} \label{eq-Wsp-Csp}
   \Cpt(E) = \inf_{u} {\|u\|^p_{\Wsp(\Rn)}},
\end{equation}
where the infimum is taken over all $u \in \Wsp(\Rn)$ such that
$0 \le u \le 1$ everywhere and $u=1$ 
in  an open set containing $E$.

If moreover 
$E=G$ is open and 
$\Csp(G)<\infty$, 
then the infimum is attained, i.e.\
there is $\psi \in \Wsp(\Rn)$ such that
$0 \le \psi \le 1$ everywhere, $\psi=1$ 
in $G$ 
and $\|\psi\|^p_{\Wsp(\Rn)} = \Cpt(G)$.
\end{lem}

\begin{proof}
By \eqref{eq-cap-E} it is enough to consider $E=G$ open.
Let $K_1\subset K_2 \subset \dots$ be compact sets such that
\[
G=\bigcup_{j=1}^\infty K_{j} \quad \text{and}  \quad \Cpt(K_j)\to \Cpt(G).
\]

Assume first that $u \in \Wsp(\Rn)$ is such that
$0 \le u \le 1$ everywhere and $u=1$ on $G$.
Then a mollification $\ub_j$ of $u$ is admissible for 
$\Cpt(K_j)$ and satisfies 
\[
\Cpt(K_j) \le \|\ub_j\|^p_{\Wsp(\Rn)} \le \|u\|^p_{\Wsp(\Rn)} + 1/j.
\]
Letting $j\to\infty$ then gives $\Cpt(G) \le \|u\|^p_{\Wsp(\Rn)}$,
which shows the $\le$ inequality in~\eqref{eq-Wsp-Csp}.
In particular \eqref{eq-Wsp-Csp} holds when $\Csp(G)=\infty$.

Assume next that $\Csp(G)<\infty$.
Then we can find $u_j \in C_c^\infty(\Rn)$ so that $u_j\ge1$ on $K_j$
and $\|u_j\|^p_{\Wsp(\Rn)} \le \Cpt(G) + 1/j$.
Let $v_j= \min\{(u_j)_\limplus,1\}$.
Since truncations decrease the norm, we see that 
\[
\|v_j\|^p_{\Wsp(\Rn)} \le \|u_j\|^p_{\Wsp(\Rn)} \le \Cpt(G) + 1/j < \infty.
\]
By the reflexivity of $\Wsp(\Rn)$, there is a 
subsequence of $\{v_j\}_{j=1}^\infty$
which converges weakly in $\Wsp(\Rn)$
to a function $\psi$
with $\|\psi\|^p_{\Wsp(\Rn)} \le \Cpt(G)$.
Using that
$u \mapsto \int_E u\,dx$ is a bounded linear functional 
on $\Wsp(\R^n)$ for every bounded measurable set $E$,  
we obtain  that $0 \le \psi \le 1$ a.e.\ and  $\psi=1$ a.e.\ in $G$.
After redefinition on a set of measure zero we get
$0 \le \psi \le 1$ everywhere on $\Rn$ and
$\psi =1$ everywhere in $G$.
Thus the  $\ge$ inequality in~\eqref{eq-Wsp-Csp} also holds and 
the infimum is attained.
\end{proof}

\begin{remark} \label{rmk-Wsp-csp-2}
It follows from the Poincar\'e inequality~\eqref{eq-PI} that
$\|\cdot\|_{\Wsp(\R^n)}$, $[\,\cdot\,]_{\Wsp(\R^n)}$ and  $[\,\cdot\,]_{\Vsp(\Om)}$
are equivalent norms on $\Vsp_0(\Om)$.
With this the proof of Lemma~\ref{lem-Wsp-Csp} can be modified to show
that for open sets $G\Subset\Om$, there is 
$\psi\in \Vsp_0(\Om)$ such that $0 \le \psi \le 1$ everywhere, 
$\psi=1$ in $G$ and 
$[\psi]^p_{\Wsp(\Rn)} \le \cpt(G,\Om)$.
Here, however, mollifications of $\psi$ typically do not
have compact support in $\Om$.
Therefore the
mollification argument at the beginning of the proof of Lemma~\ref{lem-Wsp-Csp}
does not apply, and at this
point we cannot conclude the equality $[\psi]^p_{\Wsp(\Rn)} = \cpt(G,\Om)$.
By using the Kellogg property, this will be done in
Proposition~\ref{prop-Wsp-cpt}.
\end{remark}

\section{The Kellogg property (Theorem~\ref{thm-kellogg})}
\label{sec-kellogg}

We are now ready to prove the Kellogg property, that the
set of Sobolev irregular points has capacity zero.

\begin{proof}[Proof of the Kellogg property\/ \textup(Theorem~\ref{thm-kellogg}\textup)]
Theorem~\ref{thm-Wiener} shows that $x_0 \in \Om$ is Sobolev
irregular if and only if
\[
  \int_0^1 \biggl(\frac{\csp(\itoverline{B(x_0, \rho)}\setm \Om,B(x_0,2\rho))}{\rho^{n-sp}}
     \biggr)^{1/(p-1)}       \frac{d \rho}{\rho}  <\infty.
\]    
Lemma~\ref{lem-cap-extended}, together with  Bj\"orn~\cite[Lemma~2.4]{JBWien},
shows that this is in turn equivalent to the condition
\[
\int_0^1 \biggl(\frac{\cpa(F\cap \itoverline{B'(z_0,\rho)},B'(z_0,2\rho))}
      {\cpa(B'(z_0,\rho),B'(z_0,2\rho))}
\biggr)^{1/(p-1)}    \frac{d\rho}{\rho} < \infty
\quad \text{at }z_0=(x_0,0)\in\R^{n+1},
\]
where 
\[
 F= (\R^n\setm\Om)\times\{0\},
\quad 
B'(z_0,r)=\{z \in \R^{n+1}: |z-z_0| <r\}
\]
and
$\cpa$ is the weighted \p-capacity in $\R^{n+1}$ associated with the weight
\[
  w(x,t)=|t|^a \quad \text{for }
  a=p(1-s)-1,
\]
as in  Heinonen--Kilpel\"ainen--Martio~\cite[Chapter~2]{HeKiMa}.
Hence, by Theorem~21.30  in~\cite{HeKiMa}, the point $x_0\in\bdy\Om$ is 
Sobolev irregular
for $\LL u=0$ in $\Om\subset\R^n$ if and only if $z_0=(x_0,0)$ is irregular for the
weighted \p-Laplace equation
\[
  \Div(w|\grad u|^{p-2}\grad u)=0 \quad \text{in } \R^{n+1}\setm F.
\]
The Kellogg property~\cite[Theorem~9.11]{HeKiMa}  for such equations
implies that the set $I_\Om\times\{0\}$ has weighted \p-capacity zero, i.e.\ 
\[
\cpa( I_\Om\times\{0\},B')=0
\]
for any sufficiently large open ball $B'\subset\R^{n+1}$.
Another use of 
Lemma~\ref{lem-cap-extended} now shows
that $\csp(I_\Om,B)=0$
for all sufficiently large open balls $B\subset\R^{n}$.
Hence $\Csp(I_\Om)=0$,
by Proposition~\ref{prop-cp-Cp}.
\end{proof}

The following fact may also be of interest.

\begin{prop}
Let $I_\Om\subset\bdy\Om$ be the set of Sobolev irregular boundary points for $\LL$
  in $\Om$.
Then $I_\Om$ is an $F_\sigma$ set.
\end{prop}

\begin{proof}
For each  $i=1,2,\ldots$\,, we can cover $\bdy \Om$ by a finite number of open balls
$B_{i,j}= B(x_{i,j}, 1/i)$, $1 \le j \le N_i$.
Choose
$\phi_{i,j}\in \Lip_c(2B_{i,j})$ 
such that
$0 \le \phi_{i,j} \le 1$ in $\R^n$ 
and $\phi_{i,j}=1$ on $B_{i,j}$.
Consider the sets 
\[
I_{i,j} = \Bigl\{ x \in  \itoverline{B}_{i,j} \cap\bdy \Om : 
     \liminf _{\Om \ni y \to x} H \phi_{i,j}(y) < \phi_{i,j} (x)=1\Bigr\}.
\]
Note that $I_{i,j}$ contains only Sobolev irregular points.

Next, let $x \in I_\Om$.
By Theorem~\ref{thm-Sobolev-reg} there is a 
bounded  function $g \in  \VspOm \cap C(\Rn)$ such that 
it is \emph{not} true that 
$ \lim_{\Om \ni y \to x} H g(y) = g (x)$
(i.e.\ either the limit does not exist or it exists and is different from $g(x)$).
By considering $-g$ if necessary, and adding a constant,
we can assume that $g \ge 1$ in $\R^n$  and that
\[
      \liminf_{\Om \ni y \to x} H g(y) < g (x).
\]
Since $g$ is continuous we can find a ball $B_{i,j}$ 
containing the point $x$ so that
\[
      M:= \inf_{2B_{i,j}} g 
    > \liminf_{\Om \ni y \to x} H g(y).
\]
Thus, $\phi_{i,j} \le g/M$ in $\Rn$, and hence, by 
the comparison principle (Corollary~\ref{cor-comparison-sol}),
\[
       \liminf_{\Om \ni y \to x} H \phi_{i,j}(y)
       \le \frac{1}{M} \liminf_{\Om \ni y \to x} H g(y)
        < 1 = \phi_{i,j} (x),
\]
i.e.\ $x \in I_{i,j}$.
We therefore conclude that
\begin{equation*}
    I_\Om = \bigcup_{i=1}^\infty \bigcup_{j=1}^{N_i} I_{i,j}
  = \bigcup_{i=1}^\infty \bigcup_{j=1}^{N_i} \bigcup_{k=1}^\infty 
      ( \itoverline{B}_{i,j}\cap \bdy \Om  \cap 
      \overline{\{y \in \Om: H \phi_{i,j}(y) <1-1/k\}}),
\end{equation*}
which is a countable union of compact sets, and thus an
$F_\sigma$ set.
\end{proof}

As an application of the Kellogg property we can obtain the following result,
cf.\ Remark~\ref{rmk-Wsp-csp-2}.
Here,
we let $\Et$, $\Kt$ and $\Ht$  denote the energy,
obstacle problem and Sobolev solution  with respect
to the kernel $k(x,y)=|x-y|^{-n-sp}$.

\begin{prop} \label{prop-Wsp-cpt}
Let $E\Subset\Om$.  
Then
\begin{equation*} 
   \cpt(E,\Om) = \inf_{u \in \A} {[u]_{\Wsp(\Rn)}^{p}},
\end{equation*}
where
\[
 \A=\{u \in \VspoOm: 0 \le u \le 1 \text{ in $\Rn$ and } u=1 \text{ in 
an open set containing } E\}.
\]

If moreover $E$ is open, then the infimum is attained by any
solution $v \in \A$ of the $\Kt_{\chi_G,0}(\Om)$-obstacle problem,
and there are such solutions.
\end{prop}

Note that if
$E$ is not open, 
then the infimum is usually not attained, for example not if
$\cpt(E,\Om)=0$ and $E \ne \emptyset$.
In the proof we will use the following lemma, 
whose proof is standard
and provided here for the reader's convenience.

\begin{lem} \label{lem-obst-energy-min}
Let $g\in \Vsp(\Om)$ and $\psi:\Om\to\eR$ be such that $\K_{\psi,g}(\Om) \ne \emptyset$.
Then $u$ is a solution of the $\K_{\psi,g}(\Om)$-obstacle problem 
if and only if $u$ minimizes the energy $\E(u,u)$ among all functions in 
$\K_{\psi,g}(\Om)$.
\end{lem}

\begin{proof}
Assume first that $u$ is a solution of the $\K_{\psi,g}(\Om)$-obstacle problem.
Let $v \in \K_{\psi,g}(\Om)$. 
Then, $\E(u,v)-\E(u,u)=\E(u,v-u) \ge 0$.
H\"older's inequality, then gives
\begin{equation} \label{eq-min}
 \E(u,u)  \le \E(u,v) \le  \E(u,u)^{1-1/p} \E(v,v)^{1/p}.
\end{equation}
Since $\E(u,u)<\infty$, we conclude that $\E(u,u) \le \E(v,v)$
and thus $u$ is an energy minimizer.

To show the converse implication, assume that
$u \in \K_{\psi,g}(\Om)$ is an energy minimizer within $\K_{\psi,g}(\Om)$.
To show that $u$ is a solution of the $\K_{\psi,g}(\Om)$-obstacle problem,
let $v$ be such a solution 
provided by Theorem~\ref{thm-obst-KL-solv}.
By the above, $v$ is also a minimizer and so
we must have equality throughout \eqref{eq-min},
in particular in the H\"older inequality.
Therein, equality holds if and only if the functions
$(x,y)\mapsto u(x)-u(y)$ and $(x,y)\mapsto v(x)-v(y)$
are linearly dependent.
Hence, for some 
$a,b\in \R$, $(a,b) \ne (0,0)$, 
\[
  a(u(x)-u(y))=b(v(x)-v(y)) 
\quad \text{for  a.e.\ } (x,y) \in \R^{2n}
\]
and thus $|a|^p\E(u,u)=|b|^p\E(v,v)$.
Since $\E(u,u)=\E(v,v)$ we must have $a=\pm b\ne 0$,
and using that $u(y)=v(y)=0$ for $y\notin\Om$, we get that
$u(x)=\pm v(x)$ for a.e.\ $x\in\Om$.

If $u=v$ a.e.\ in $\Om$, then clearly, $u$ is a solution 
of the $\K_{\psi,g}(\Om)$-obstacle problem.
On the other hand, if $u=-v$,  then $0=\tfrac12(u+v) \in \K_{\psi,g}(\Om)$ and 
so $\E(u,u)\le \E(0,0)=0$, i.e.\ $u=-v=0$ a.e.\ is
also  a solution of the $\K_{\psi,g}(\Om)$-obstacle problem.
\end{proof}

\begin{proof}[Proof of Proposition~\ref{prop-Wsp-cpt}]
By \eqref{eq-cap-E}, it is enough to consider
$E=G$ open.
It follows from the Poincar\'e inequality~\eqref{eq-PI} that
$\|\cdot\|_{\Wsp(\R^n)}$, $[\,\cdot\,]_{\Wsp(\R^n)}$ and  $[\,\cdot\,]_{\Vsp(\Om)}$
are equivalent norms on $\Vsp_0(\Om)$.
With this, the proof of Lemma~\ref{lem-Wsp-Csp} can be modified to show
that there is  $\psi\in \A$ such that $[\psi]^p_{\Wsp(\Rn)} \le \cpt(G,\Om)$.

By Theorem~\ref{thm-obst-KL-solv}, there
is a solution $v$ of the $\Kt_{\chi_G,0}(\Om)$-obstacle problem.
Since $\psi \in \A \subset \Kt_{\chi_G,0}(\Om)$
it  follows from Lemma~\ref{lem-obst-energy-min} that
\[
[v]_{\Wsp(\Rn)}^{p}
= \Et(v,v) \le \Et(\psi,\psi) 
= [\psi]_{\Wsp(\Rn)}^{p}
\le \cpt(G,\Om).
\]
The comparison principle for obstacle problems
 (Theorem~\ref{thm-comp-obs})
implies that $v\le1$ a.e.\ in $\Om$ and 
after redefinition on a set of measure zero we can assume
that $v = 1$ everywhere in $G$.

It remains to show that 
$\cpt(G,\Om) \le [v]_{\Wsp(\Rn)}^{p}$.
The function $v$ is, by the definition of the obstacle problem,
also a solution  of the $\Kt_{0,v}(\Om \setm \clG)$-obstacle problem.
By the comparison principle (Corollary~\ref{cor-comparison-sol}),
the Sobolev solution
$\Ht_{\Om \setm \clG} v \ge 0$ 
in $\Om \setm \clG$ and hence  it is also a solution of the 
$\Kt_{0,v}(\Om \setm \clG)$-obstacle problem.
By the uniqueness for solutions of obstacle problems, we conclude that
$v=\Ht_{\Om \setm \clG} v$ a.e., and we can replace $v$ by $\Ht_{\Om \setm \clG} v$,
which we do 
in the rest of the proof.

Let $I_\Om$ be the set of Sobolev irregular boundary points for $\Om$
and $0<\eps<\dist(G,\Om^c)$.
By the Kellogg property (Theorem~\ref{thm-kellogg})
and Proposition~\ref{prop-cp-Cp}, 
\[
\cpt(I_\Om,V_\eps) = \Cpt(I_\Om)=0,
\]
where $V_\eps$ is the  $\eps$-neighbourhood of $\bdy\Om$.
Thus, by Remark~\ref{rmk-Wsp-csp-2}, there
is $\phi \in \Vsp_0(V_\eps)$ such that
$0 \le \phi \le 1$ everywhere, $\phi=1$ in an open set $V \supset I_\Om$ and 
$[\phi]_{\Wsp(\Rn)} <\eps$.
With $u:=(v-\phi)_\limplus \ge 0$ we then get from
Theorem~\ref{thm-Sobolev-reg} and Corollary~\ref{cor-local-subset} that
\[
   0 
  \le    \limsup_{\Om \ni x \to x_0} u(x)
   \le  \lim_{\Om \ni x \to x_0} v(x) 
= \lim_{\Om \ni x \to x_0} \Ht_{\Om \setm \clG} v(x) 
 =0 
\quad \text{if } x_0 \in \bdy \Om \setm I_\Om.
\]
On the other hand, since $\phi=1$ in $V$, 
\[
    \lim_{\Om \ni x \to x_0} u(x)=0 
\quad \text{if } x_0 \in I_\Om.
\]
Thus $u_\eps:=((1+\eps)u-\eps)_\limplus$ has compact support in $\Om$,
and so it follows from a standard mollification argument that
\begin{equation*}
\cpt(G,\Om)  \le     [u_\eps]_{\Wsp(\Rn)}^p
\le (1+\eps)^p [u]_{\Wsp(\Rn)}^p. 
\end{equation*}
Since also
\[
[u]_{\Wsp(\Rn)} \le  [v-\phi]_{\Wsp(\Rn)}  
<  [v]_{\Wsp(\Rn)}  + \eps,
\]
letting $\eps \to 0$ shows that  indeed 
$    \cpt(G,\Om) \le   [v]_{\Wsp(\Rn)}^p$.
\end{proof}

\section{A convergence result for the obstacle problem}\label{sec-obstacle}

Next, we discuss the pointwise convergence of sequences of solutions of the obstacle problem and the Dirichlet problem.

\begin{thm}   \label{thm-conv-obst-prob}
Assume that $g_j\in \Vsp(\Om)$, $j=0,1,\ldots$\,,  and that $g_j\to g_0$  in $\Vsp(\Om)$ 
and $g_j\searrow g_0$ pointwise in $\R^n$, as $j \to \infty$.
Let $u_j$ be a solution of the
$\K_{g_j,g_j}(\Om)$-obstacle problem, $j=0,1,\ldots$\,.
Then $u_j \searrow u_0$  a.e.\ in $\Om$.

Moreover, $Hg_j \searrow Hg_0$ everywhere in $\Om$.
\end{thm}

For the Dirichlet problem in the last part 
of Theorem~\ref{thm-conv-obst-prob}
we get convergence everywhere.
On the contrary, the following example shows that
for the obstacle problem the pointwise convergence 
$u_j(x) \to u_0(x)$ can fail at some points, even if all $u_j$, $j=0,1,\ldots$\,,
are continuous.

\begin{example}
Assume that $0 <sp \le n$.
Let $\Om=B(0,2)$, $K_j=\itoverline{B(0,1/j)}$ and 
$g_j=\pot(K_j,\Om)$ be the corresponding $\LL$-potentials, $j=1,2,\ldots$\,,
as in Definition~\ref{def-pot}.
Then each boundary point of $\Om \setm K_j$ satisfies
the exterior cone condition and is thus Sobolev regular
by the Wiener criterion (Theorem~\ref{thm-Wiener}).
Hence $g_j$ is continuous in $\R^n$.
Moreover, 
$\{g_j\}_{j=1}^\infty$ is a pointwise decreasing sequence,
by the comparison principle (Corollary~\ref{cor-comparison-sol}).
Let $g_0=\lim_{j \to \infty} g_j$.

Next, for every function $u\in C^\infty_c(\Om)$, which is admissible 
for $\csp(K_j,\Om)$, 
Lemma~\ref{lem-Vsp-C} shows that the function $\ut:=\min\{u,1\}$ satisfies
\[
\ut - g_j \in \Vsp_0(\Om \setm K_j),
\]
and so $\ut \in \K_{-\infty,g_j}(\Om \setm K_j)$.
It thus follows from Lemma~\ref{lem-obst-energy-min},
with $\La\ge1$ as in~\eqref{eq-comp-(x,y)}, that
\[
\La^{-1}  [g_j]^p_{W^{s,p}(\R^n)} 
\le  \E(g_j,g_j) \le \E(\ut,\ut)
\le \La  [\ut]^p_{W^{s,p}(\R^n)}
\le \La [u]^p_{W^{s,p}(\R^n)}.
\]
Taking infimum over all such $u$ shows that
\[
 0 \le \limsup_{j \to \infty} {[g_j]_{W^{s,p}(\R^n)}^p} 
   \le \La^2   \lim_{j \to \infty} \cpt(K_j,\Om)
    =  \La^2
   \csp(\{0\},\Om) =0,
\]
by Lemma~\ref{lem-sp<=n}. (Here we use that $sp \le n$.)
It then follows from the Poincar\'e inequality~\eqref{eq-PI} that also
$\lim_{j \to \infty} \| g_j\|_{L^p(\R^n)}=0$, and thus $g_0=0$ a.e.

Clearly $u_j:=g_j$ is the continuous solution of the
$\K_{g_j,g_j}(\Om)$-obstacle problem, $j=1,2,\ldots$\,.
On the other hand, $u_0 \equiv 0$ is the continuous solution of the
$\K_{g_0,g_0}(\Om)$-obstacle problem.
Thus $u_j(0) \searrow g_0(0)=1 > 0=u_0(0)$,
i.e.\ the convergence is not everywhere.
In fact, it follows from Harnack's convergence theorem 
(Theorem~\ref{thm-harnack} below)
that $g_0=\chi_{\{0\}}$.
\end{example}

\begin{proof}[Proof of Theorem~\ref{thm-conv-obst-prob}]
The sequence $\{u_j\}_{j=1}^\infty$ is decreasing a.e.\ in $\Om$ by Theorem~\ref{thm-comp-obs}.
Write $g=g_0$ and define 
\[
u=\liminf_{j\to\infty}u_j.
\]
In particular, $u=\lim_{j \to \infty} g_j=g$ in~$\Omc$.
It is enough to show that $u$ is a solution of 
the $\mathcal{K}_{g, g}(\Omega)$-obstacle problem 
because the uniqueness of solutions to the obstacle problem 
(see Theorem~\ref{thm-obst-KL-solv})
then yields $u=u_0$ a.e.\ in $\mathbb{R}^n$.

We first claim that $u \in \mathcal{K}_{g, g}(\Omega)$. 
Indeed, from $Hg \le u \le u_j \le u_1$ a.e.\  we see that
$u \in L^p(\Omega)$.
In order to prove that $u \in \VspOm$, we define 
$d\mu(x, y)=k(x, y) \,dy\,dx$ and introduce
\begin{equation*}
w_j(x, y) = |u_j(x)-u_j(y)|^{p-2}(u_j(x)-u_j(y)), \quad j=1, 2,\dots.
\end{equation*}
Note that
\begin{equation}\label{eq-uw-comp}
\La^{-1} [u_j]_{\Vsp(\Om)}^p \le \|w_j\|_{L^{p'}((\Om^c\times \Om^c)^c, \mu)}^{p'} 
\le 2\La  [u_j]_{\Vsp(\Om)}^p,
\end{equation}
where $p'=p/(p-1)$ and
$\La\ge1$ is as in~\eqref{eq-comp-(x,y)}.
Since $u_j$ is a solution of the $\K_{g_j,g_j}(\Om)$-obstacle problem 
and $g_j \in \K_{g_j, g_j}(\Om)$, 
it follows from Lemma~\ref{lem-obst-energy-min}
that
\begin{equation}\label{eq-E-uj-bdd}
\|w_j\|_{L^{p'}((\Om^c\times \Om^c)^c, \mu)} 
= \E(u_j,u_j)^{1/p'} \le \E(g_j,g_j)^{1/p'}
\leq (2\Lambda)^{1-1/p} [g_j]_{\VspOm}^{p-1}.
\end{equation}
In particular, \eqref{eq-uw-comp} and \eqref{eq-E-uj-bdd} show that
$[u]_{\Vsp(\Om)} \le \liminf_{j\to\infty} [u_j]_{\Vsp(\Om)} < \infty$, 
which gives $u\in \Vsp(\Om)$.
Next, $0 \le (u-g_j)_\limplus \le (u_j-g_j)_\limplus$ a.e.\ and
$(u_j-g_j)_\limplus \in \Vsp_0(\Om)$
by Lemma~\ref{lem-Vspo-limplus}.
Thus Corollary~\ref{cor-police} shows that 
$(u-g_j)_\limplus \in \Vsp_0(\Om)$ for all $j=1,2,\ldots$\,.
Lemma~\ref{lem-truncation},
 together with the fact that $u-g_j\to u-g$ 
in $\Vsp(\Om)$ as $j\to\infty$, yields 
\[
\Vsp_0(\Om) \ni (u-g_j)_\limplus \to (u-g)_\limplus \quad \text{in }\Vsp(\Om),
\]
and hence $(u-g)_\limplus \in \Vsp_0(\Om)$ 
(because $\Vsp_0(\Om)$ is complete by definition).
Since clearly $u\ge g$ a.e.\ in $\Om$ and $u=g$ in $\Om^c$, 
this implies that 
$u-g \in \Vsp_0(\Om)$ and so
$u \in \mathcal{K}_{g, g}(\Om)$.

Next we prove that $u$ is a solution of the $\mathcal{K}_{g, g}(\Omega)$-obstacle
problem.
For every $v \in \K_{g,g}(\Om)$, we have $v+g_j-g\in\K_{g_j,g_j}(\Om)$
and hence $\E(u_j,v+g_j-g-u_j)\ge0$, which yields
\begin{align}
&\iint_{(\Om^c\times\Om^c)^c} |u_j(x)-u_j(y)|^p \,d\mu \nonumber \\
&\qquad \leq \iint_{(\Om^c \times \Om^c)^c} w_j(x, y) (v(x)-v(y)) \,d\mu \label{eq-int-with-v} \\
&\qquad \quad + \iint_{(\Om^c \times \Om^c)^c} w_j(x, y) ((g_j-g)(x)-(g_j-g)(y)) \,d\mu.
\nonumber
\end{align}
It follows from \eqref{eq-E-uj-bdd} that $\{w_j\}_{j=1}^\infty$ is a bounded sequence 
in $L^{p'}((\Om^c \times \Om^c)^c, \mu)$.
We can therefore find a subsequence, also denoted by $w_j$, weakly converging to $w$ 
in $L^{p'}((\Om^c \times \Om^c)^c, \mu)$, such that
\begin{equation*}
w_j(x,y) \to w(x,y) := |u(x)-u(y)|^{p-2} (u(x)-u(y))
\quad \text{$\mu$-a.e.\ as } j\to \infty.
\end{equation*}
The weak convergence of $w_j$ in $L^{p'}((\Om^c \times \Om^c)^c, \mu)$ implies that
\begin{equation*}
\iint_{(\Om^c \times \Om^c)^c} w_j(x, y)(v(x)-v(y)) \,d\mu 
\to \iint_{(\Om^c \times \Om^c)^c} w(x, y) (v(x)-v(y)) \,d\mu.
\end{equation*}
At the same time, by H\"older's inequality,
\begin{equation*}
\iint_{(\Om^c \times \Om^c)^c} w_j(x, y) ((g_j-g)(x)-(g_j-g)(y)) \,d\mu 
\leq 2\La [u_j]_{\Vsp(\Om)}^{p-1} [g_j-g]_{\Vsp(\Om)} \to 0,
\end{equation*}
since $\{u_j\}_{j=1}^\infty$ is a bounded sequence in $\Vsp(\Om)$
and $g_j\to g$ in $\Vsp(\Om)$.
Inserting the last two estimates into~\eqref{eq-int-with-v}, 
we therefore obtain that
\begin{align*}
\iint_{(\Om^c\times\Om^c)^c} |u(x)-u(y)|^p \, d\mu
&\leq \liminf_{j\to\infty} \iint_{(\Om^c\times\Om^c)^c}
|u_j(x)-u_j(y)|^p\, d\mu \\
&\le \iint_{(\Om^c\times\Om^c)^c} w(x, y) (v(x)-v(y)) \,d\mu,
\end{align*}
from which we conclude that $\mathcal{E}(u, v-u) \ge0$.
Thus $u$ is a  solution of the $\K_{g,g}(\Om)$-obstacle problem.

The convergence $Hg_j \to Hg$ a.e.\ is proved in the same way, just replacing
$\K_{g,g}(\Om)$ and $\K_{g_j,g_j}(\Om)$ by $\K_{-\infty,g}(\Om)$ and
$\K_{-\infty,g_j}(\Om)$.
As $Hg_j$ and $Hg$ are continuous, we also know
that $\{Hg_j\}_{j=1}^\infty$ is a decreasing sequence of 
$\LL$-harmonic functions in $\Om$.
It then follows from
Harnack's convergence theorem below 
(applied to $\{-Hg_j\}_{j=1}^\infty$),
that there is an $\LL$-harmonic function $u$ in $\Om$
such that $Hg_j \searrow u$ everywhere.
In particular, $u \equiv g$ in $\Omc$ and $u=Hg$ a.e.\ in $\Om$.
Since both $u$ and $Hg$ are continuous in $\Om$, they must be identical.
\end{proof}

It remains to state the Harnack's convergence theorem, which we
will also need later on.

\begin{thm}\label{thm-harnack}
\textup{(Harnack's convergence theorem, 
Korvenp\"a\"a--Kuusi--Palatucci~\cite[Theorem~15]{KKP17})}
Let $\{u_j\}_{j=1}^\infty$ be an increasing sequence of functions 
on $\R^n$ such that $u_j$ is $\LL$-harmonic  in $\Om$, $j=1,2,\dots$\,.
Then either $u:=\lim_{j \to \infty} u_j$ is $\LL$-harmonic  in $\Om$ 
or $u \equiv \infty$ in $\Om$.
\end{thm}

\section{\texorpdfstring{$\LL$}{L}-superharmonic functions}\label{sec-superharmonic}

In order to define Perron solutions we need $\LL$-superharmonic functions,
which we next define.
We follow Korvenp\"a\"a--Kuusi--Palatucci~\cite[Definition~1]{KKP17}.

\begin{deff} \label{def:superharmonic}
A measurable function $u: \R^n \to \eR$ is \emph{$\LL$-superharmonic}  in $\Omega$ if it satisfies the following properties:
\begin{enumerate}
\item
$u < \infty$ almost everywhere in $\R^{n}$ and $u>-\infty$ everywhere in $\Omega$,
\item
$u$ is lower semicontinuous  in $\Omega$,
\item
for each open set $G \Subset \Omega$ and each solution $v \in C(\clG)$ of $\mathcal{L}v=0$ 
in $G$ satisfying $v_\limplus \in L^{\infty}(\R^{n})$ and $v \le u$ on  $\Gc$,
it holds that $v \le u$ in $G$,
\item
$u_\limminus \in L^{p-1}_{sp}(\R^{n})$.
\end{enumerate}
A function $u$ is  \emph{$\LL$-subharmonic} in $\Omega$ if $-u$ is $\LL$-superharmonic in $\Omega$.
\end{deff}

\begin{remark}
The measurability assumption in Definition~\ref{def:superharmonic} 
was not explicitly included in~\cite[Definition~1]{KKP17}.
However, it  was used implicitly since 
without measurability, $\chi_E$ would be $\LL$-superharmonic
for every nonmeasurable $E \subset \Om^c$  with zero inner measure,
which is undesirable and would violate parts of Theorem~\ref{thm:KKP17} below.

In \cite{KKP17} it is only required that 
$u \ge v$ everywhere on $\bdy G$ and a.e.\ on $\Rn \setm \clG$,
but since one can modify $v$ arbitrarily on a set of measure zero in
$\Rn \setm \clG$, our definition is equivalent to theirs.
Our formulation will turn out to be convenient later.
\end{remark}

It follows directly that the pointwise minimum of two $\LL$-superharmonic functions
is also $\LL$-superharmonic.
Moreover, a function
is $\LL$-harmonic in $\Omega$ if and only if it  is both $\LL$-superharmonic and 
$\LL$-subharmonic in $\Omega$, see Corollary~7 in~\cite{KKP17}.

The following result from \cite{KKP17}
explains the close connection between supersolutions and $\LL$-superharmonic functions.
Recall that the lsc-regularization $\uhat$ of $u$ was defined in~\eqref{eq-def-uhat}.

\begin{thm} 
\label{thm:KKP17}
\textup{(\cite[Theorems~1, 9 and~12]{KKP17})}
If $u$ is $\LL$-superharmonic  in $\Omega$, then $u$ is lsc-regularized in $\Om$.
If, in addition, $u$ is locally bounded in $\Omega$ or $u \in W^{s, p}_{\mathrm{loc}}(\Omega)$, then it is a  supersolution in $\Omega$.

Conversely, if $u$ is a supersolution in $\Om$, then $\uhat=u$ a.e.\
and $\uhat$ is $\LL$-superharmonic in $\Om$.
\end{thm}

The following comparison principle is
essential in this paper.
In  particular, it implies the fundamental inequality
$\lP f \le \uP f$ for Perron solutions.

\begin{thm} \label{thm-comp}
\textup{(Comparison principle, \cite[Theorem~16]{KKP17})}
Assume that $u$ is $\LL$-superharmonic in $\Om$ 
and $v$ is $\LL$-subharmonic in $\Om$.
If $v \le u$ a.e.\ in $\Omc$ and
\begin{equation*}
       \infty \ne  \limsup_{\Om \ni y \to x_0} v(y)
        \le \liminf_{\Om \ni y \to x_0} u(y) \ne -\infty
        \quad \text{for all } x_0 \in \bdy \Om,
\end{equation*}
then $v \le u$ in $\Om$.
\end{thm}

We will also need the following result about Poisson modifications.

\begin{thm} \label{thm-Poisson}
\textup{(Poisson modification, \cite[Lemma~15 and Theorem~18]{KKP17})}
Let $G \Subset \Om$ be an open subset 
satisfying the following measure
density condition\/\textup:
\begin{equation} \label{eq-meas-dens}
      \inf_{x_0 \in \bdy G}  \inf_{0 < r< r_0}  
\frac{|B(x_0,r) \setm G|}{|B(x_0,r)|} >0 
\quad \text{for some } r_0>0.
\end{equation}

If $u$ is an $\LL$-superharmonic function in $\Om$,
then there is a unique function $\ut_G$, called the Poisson modification of $u$ in $G$, 
such that
$\ut_G$ is  $\LL$-harmonic in $G$,
$\LL$-superharmonic in $\Om$,
$\ut_G \le u$ in $G$ and  $\ut_G=u$ outside $G$.

Moreover, if $u$ and $v$ are $\LL$-superharmonic functions in $\Om$ such that
$u \le v$ in $\R^n$, and $G_1 \subset G_2 \Subset \Om$ are open subsets satisfying
the measure density condition~\eqref{eq-meas-dens}, 
then $\ut_{G_2} \le \vt_{G_1}$ in $\Rn$.
\end{thm}

\begin{proof}
That $\ut_G$ is $\LL$-harmonic in $G$ follows from the proof 
of~\cite[Theorem~18]{KKP17} since therein $P_{u,D} \equiv w$ in $\R^n$.
The rest of the conclusions in the first part were shown 
in~\cite[Theorem~18]{KKP17}.

Finally, the inequality  $\ut_{G_1} \le \vt_{G_1}$ was shown in~\cite[Lemma~15]{KKP17}.
The
general case then follows since
\[
    \ut_{G_2} = (\ut_{G_1})_{G_2} \le \ut_{G_1} \le \vt_{G_1},
\]
by the uniqueness and inequality in the first part of the theorem.
\end{proof}

Also the following pasting lemma will be needed.

\begin{lem} \label{lem-ut-Omc-change}
Let $u$ be an $\LL$-superharmonic function in $\Om$ and $g \in L^{p-1}_{sp}(\R^{n})$.
If $g \le u$ in $\Omc$, then 
\[
    \ut=\begin{cases}
     u & \text{in } \Om, \\
     g & \text{in } \Omc, \\
   \end{cases}
\]
is $\LL$-superharmonic  in $\Om$.
\end{lem}

\begin{proof}
That $\ut$ satisfies the properties (a), (b) and (d) in the Definition~\ref{def:superharmonic} of
$\LL$-superharmonic functions is clear.
As for (c), let  
$G \Subset \Omega$ be open and  $v \in C(\clG)$ be a solution of $\mathcal{L}v=0$ 
in $G$ such that $\ut \geq v$ on  $\Gc$ and $v_\limplus \in L^{\infty}(\R^{n})$.
Then $u \ge \ut \ge v$ on $\Gc$ and hence
$\ut=u \ge v$ in $G$, since $u$ is  $\LL$-superharmonic in $\Om$.
Thus (c) holds for $\ut$, and therefore $\ut$ is $\LL$-superharmonic in $\Om$.
\end{proof}

\section{Perron solutions}\label{sec-perron}

Recall Definition~\ref{def-Perron} of Perron solutions from the introduction.

Note that 
if $g$ is nonmeasurable, then the upper Perron solution $\uP g$ 
cannot be $\LL$-harmonic  in $\Om$
because of \eqref{eq-uP=g} and Definition~\ref{def-supersol}.
We can only   show $\LL$-harmonicity
when $g$ belongs to a tail space
in Theorem~\ref{thm-Perron-meas} below.
Thus in our setting, it might be more appropriate to call $\uP g$ an upper Perron function 
rather than an upper Perron solution.
However, we will mainly be interested in measurable functions, so we 
have chosen the more traditional name Perron solutions.
For general functions $g$, we obtain at least the following continuity result.
This is Theorem~\ref{thm-harmonicity}\ref{thm-harmonicity-a}.

\begin{thm} \label{thm-Perron}
Let $g : \Omc \to \eR$. 
Then $\uP g$ is either
continuous in $\Om$ or identically\/ $\pm\infty$ in $\Om$.
In particular, $\uP g$ is always 
$\eR$-continuous in $\Om$.
\end{thm}

The proof shows that 
if $\uP g$ is not identically\/ $\pm\infty$ in $\Om$,
then
$\uP g=u$ in $\Om$ for some 
function $u$ which is $\LL$-harmonic in $\Om$.
However, it is in general only true that $\uP g\le u$ in $\Om^c$ and so $\uP g$ need not
be $\LL$-harmonic in $\Om$.
Indeed, $\uP g$ can never be $\LL$-harmonic in $\Om$
for nonmeasurable $g$, cf.\ \eqref{eq-uP=g} and Proposition~\ref{prop-Perron-nonmeas}.

For measurable $g$ we get the following result,
which in particular relates our Perron solutions to those in 
Korvenp\"a\"a--Kuusi--Palatucci~\cite{KKP17}, cf.\ Theorem~2 therein.
Here, the space $L^{p-1}_{sp}(\Omega^c)$ is defined by replacing 
$\Rn$ in \eqref{eq-tail} by $\Omega^c$.
This contains Theorem~\ref{thm-harmonicity}\ref{thm-harmonicity-b}.

\begin{thm} \label{thm-Perron-meas}
Assume that  $g \in L^{p-1}_{sp}(\Omega^c)$.
Let
\[ 
    \uQ g = \inf_{u \in \UUt_g}  u, 
   \quad  \text{where }
   \UUt_g=\{u \in \UU_g : u =g \text{ on } \Omc\}.
\]
Then the following hold\/\textup:
\begin{enumerate}
\item \label{QQ-a}
$\uP g = \uQ g$ in $\Rn$.
\item \label{QQ-b}
$\uP g$ is either $\LL$-harmonic in $\Om$ or
\[
\uP g \equiv \infty \quad \text{in } \Rn
\qquad \text{or} \qquad 
     \uP g = \begin{cases}
      - \infty & \text{in } \Om, \\
      g & \text{in } \Omc.
   \end{cases}
\]
\item \label{QQ-c}
If $g_1,g_2 \in L^{p-1}_{sp}(\Omega^c)$ and $g_1 \le g_2$ in $\Om^c$,
then $\uQ g_1 \le \uQ g_2$ in $\Rn$.
\end{enumerate}
\end{thm}

In the local setting, \ref{QQ-b} requires that $\Om$ is connected.
This is not needed here because of the nonlocal nature of the operator $\LL$.
Similarly, $\Om$ is not required to be connected in 
Harnack's convergence theorem (Theorem~\ref{thm-harnack}).

Part~\ref{QQ-c} follows directly from \ref{QQ-a} together with 
\eqref{eq-Pg1-Pg2}.
Note however that this fundamental inequality 
did not appear
in~\cite{KKP17}.

\begin{proof}[Proof of Theorem~\ref{thm-Perron}]
If $\UU_g=\emptyset$, then $\uP g \equiv \infty$.
Assume therefore that $\UU_g \ne \emptyset$.
Let $G \Subset \Om$ be an open polyhedron.
Then $G$ satisfies the measure density condition~\eqref{eq-meas-dens}.

Let $v \in \UU_g$ be arbitrary
and let $\vt_G$ be the Poisson modification of $v$ in $G$ as given by
Theorem~\ref{thm-Poisson}.
As $\vt_G \le v$ in $G$ and $\vt_G = v$ in $G^c$,
we see that
$\uP g = \inf_{v\in\UU_g} \vt_G$.
Moreover, each $\vt_G$ is $\LL$-harmonic and thus continuous in $G$.
Hence the upper Perron solution $\uP g$ is upper 
semicontinuous in $G$.
Since $G \Subset \Om$ was an arbitrary polyhedron, it follows that $\uP g$ is upper
semicontinuous in $\Om$.

Let $Z = \{z_1, z_2, \ldots\}$ be a countable dense subset of $\Om$ and
for each $j=1,2,\ldots$\,, find $\LL$-superharmonic functions 
$u_{j,k} \in \UU_g$  so that
$
\lim_{k\to\infty} u_{j,k} (z_j) = \uP g (z_j).
$
As the minimum of two $\LL$-superharmonic functions is also 
$\LL$-superharmonic, 
the sequence 
\[
u_j=\min_{i,k\le j} u_{i,k}
\]
belongs to $\UU_g$, is pointwise decreasing  and has limit
\[
\lim_{j\to\infty} u_{j} (z) = \uP g (z) \quad \text{for all }z\in Z.
\]

Next let $G_1 \subset G_2 \subset \dots \subset \Om = \bigcup_{i=j}^\infty G_j$
be an exhaustion of $\Om$ by open polyhedra.
Let $\ut_j=\ut_{j,G_j}$ be the Poisson modification of $u_j$ in $G_j$ according to
Theorem~\ref{thm-Poisson}.
Then $\ut_j$ is $\LL$-harmonic in $G_j$, $\LL$-superharmonic in $\Om$
and $\{\ut_j\}_{j=1}^\infty$ is a decreasing sequence.
Let $u = \lim_{j\to\infty} \ut_j$.
As $\uP g \le \ut_j$ for every
$j$, we get that $\uP g \le u$ in $\R^n$.
At the same time, we have that 
\[
\uP g(z) = \lim_{j\to\infty} u_j(z) 
      \ge \lim_{j\to\infty} \ut_{j}(z) = u(z) 
\quad \text{for all } z \in Z,
\]
i.e.\ $u = \uP g$ on $Z$.
Moreover, 
applying 
Harnack's convergence theorem (Theorem~\ref{thm-harnack})
to $\{-\ut_j\}_{j=1}^\infty$
shows that $u$ is $\LL$-harmonic or identically $-\infty$ in every $G_j$ and thus
in~$\Om$.

If $u \equiv -\infty$ in $\Om$, then also $\uP g \equiv -\infty$ in $\Om$.
Otherwise, $u$ is continuous in $\Om$ and using 
the upper semicontinuity of $\uP g$, we find that
\[
    u(x) \ge \uP g(x) \ge \limsup_{Z \ni z \to x} \uP g(z)
     = \limsup_{Z \ni z \to x} u(z) = u(x)
     \quad \text{for } x \in \Om,
\]
i.e.\ $\uP g=u$ is continuous 
in $\Om$.
\end{proof}

\begin{proof}[Proof of Theorem~\ref{thm-Perron-meas}]
\ref{QQ-a}
For each  $u \in \UU_g$, let 
\[
    \ut=\begin{cases}
     u & \text{in } \Om, \\
     g & \text{in } \Omc, \\
   \end{cases} 
\]
which is $\LL$-superharmonic in $\Om$, by Lemma~\ref{lem-ut-Omc-change}.
Since $u \in \UU_g$, it is straightforward to see that $\ut \in \UUt_g$
and that $\ut \le u$ in $\Rn$.
It follows
that 
\[ 
    \uP g 
  = \inf_{u \in \UU_g}  u
  = \inf_{u \in \UU_g}  \ut 
  = \inf_{v \in \UUt_g}  v
  = \uQ g.
\]

\ref{QQ-b}
If  $\UUt_g = \emptyset$, then $\UU_g = \emptyset$ (by the proof of \ref{QQ-a}) 
and thus
$\uP g \equiv \infty$ in $\Rn$.
So assume that $\UUt_g \ne \emptyset$.
Since $\uP g = \uQ g$ by \ref{QQ-a}, 
we can rerun the proof of Theorem~\ref{thm-Perron} with $u_{j,k} \in \UUt_g$.
Let $u$ be the function thus constructed.
Then $u=\uP g$ in $\Om$ and $u$ is 
$\LL$-harmonic or identically $-\infty$ in $\Om$,
by the proof of Theorem~\ref{thm-Perron}.
Moreover, each $u_{j,k}=g$ in $\Omc$, and thus also $u=g=\uP g$
in $\Omc$, i.e.\ $u=\uP g$ in $\R^n$.
Hence $\uP g$ is 
$\LL$-harmonic or identically $-\infty$ in $\Om$,
which together with \eqref{eq-uP=g} completes the proof of \ref{QQ-b}.

\ref{QQ-c}
This follows directly from \ref{QQ-a} together with 
\eqref{eq-Pg1-Pg2}.
\end{proof}

It is easy to see that every $g:\Omc \to \eR$ has an a.e.-unique 
\emph{minimal measurable majorant} 
$\gt:\Omc \to \eR$ in the sense that
$\gt\le f$ a.e.\ whenever $f\ge g$ is  measurable.
Indeed, for a bounded function $g$ and a bounded measurable set $E\subset\Om^c$, 
consider a minimizing sequence $\{f_j\}_{j=1}^\infty$ for the functional
$f \mapsto \int_{E} f\,dx$
in the class of
all measurable functions $f\ge g$, and set $\gt : =\inf_j f_j$ in $E$.
Then cover $\Om^c$ by countably many pairwise disjoint bounded measurable sets
and approximate general $g$ by their bounded truncations from above and below.

If $|\bdy \Om|=0$, then we can always find a minimal measurable majorant $\gt$ of
$g$ such that $\gt=g$ on $\bdy \Om$.
When $|\bdy \Om|>0$ this is not possible for every $g$.
Nevertheless the following result is relevant also when $|\bdy \Om|>0$.

\begin{prop} \label{prop-Perron-nonmeas}
Let $g:\Omc \to \eR$
and assume that $\gt$ is
a minimal measurable majorant of $g$
such that $\gt=g$ on $\bdy \Om$.
Then $\uP g = \uP \gt$ in $\Om$.
\end{prop}

\begin{proof}
Let $Z = \{z_1, z_2, \ldots\}$ be a countable dense subset of $\Om$.
As in the proof of Theorem~\ref{thm-Perron}, we can construct 
a decreasing sequence of functions $u_j \in \UU_g$
such that
\[
\lim_{j\to\infty} u_{j} (z) = \uP g (z) \quad \text{for all }z\in Z.
\]
Since each $u_j$ is measurable, we see that $u_j \ge \gt$ a.e.\ in $\Omc$.
Let 
\[
     \ut_j=\begin{cases}
        u_j & \text{in } \Om, \\
        \max\{u_j,\gt\} & \text{in } \Omc.
    \end{cases}
\]
As $\gt=g$ on $\bdy \Om$, it follows that 
$\ut_j \in \UU_{\gt}$.
Hence we get  that 
\[
\uP g(z) = \lim_{j\to\infty} \ut_{j} (z) \ge  \uP \gt (z) \ge
\uP g(z) \quad \text{for all }z\in Z,
\]
i.e.\ $\uP g=\uP \gt$ in $Z$.
Since both $\uP g$ and $\uP \gt$ are $\eR$-continuous in $\Om$, by 
Theorem~\ref{thm-Perron}, we see that $\uP g=\uP \gt$ in $\Om$.
\end{proof}

We next prove resolutivity results and the
equality between Perron and Sobolev solutions 
for a large class of boundary data.
This result contains Theorem~\ref{thm-main-Hg=Pg}
as an important special case.

\begin{thm}   \label{thm-Hg=Pg}
Let $g \in \VspOm$ be such that either $g \in C(\Rn)$,
or $g$ is bounded in $\Rn$ and
continuous at all points in $\bdy \Om \setm E$
for some set $E$ with $\Csp(E)=0$.
Also let $f=g+h$, where  $h:\R^n \to \eR$ and 
$h=0$ in $\Ec$.
Then 
\[
     Pg =Hg
\quad \text{and} \quad
     Pf = Hf = Pg + h\chi_{\Omc}.
\]
In particular, $f$ and $g$ are resolutive and $Pf$ and $Pg$ are  $\LL$-harmonic
in $\Om$.
\end{thm}

\begin{proof}
Assume to start with that $g \ge 0$.
Let $I_\Om$ be the set of Sobolev irregular boundary points for $\Om$.
By the Kellogg property (Theorem~\ref{thm-kellogg}), 
$\Csp(I_\Om)=0$.
Let $E'=I_\Om \cup E$.
Then also $\Csp(E')=0$.
By the definition of $\Csp(E')$,  there are open sets
$G_j\supset E'$ such that $\Csp(G_j)< 2^{-jp}$, $j=1,2,\ldots$\,.
For each $G_j$, it follows from Lemma~\ref{lem-Wsp-Csp} that
there is 
a function $\psi_j\ge0$ such that
$\psi_j= 1$ on $G_j$ and $\|\psi_j\|_{\Wsp(\R^n)}< 2^{-j}$.
Let 
\[
\phi_j= \sum_{i=j}^{\infty} \psi_i.
\]
Note that $\phi_j=\infty$ on $E'$, 
$\|\phi_j\|_{\Wsp(\R^n)}< 2^{1-j}$,
and $\phi_j(x) \searrow 0$ for a.e.\ $x \in \R^n$.

Since $|E|=0$ by Lemma~\ref{lem-zero-cap}, we see that $f \in \VspOm$. 
Let $f_j= Hf+\phi_j$ and let $u_j$ be the lsc-regularized solution 
of the $\K_{f_j,f_j}(\Om)$-obstacle problem.
Then each $u_j$ is $\LL$-superharmonic in $\Om$ by Theorem~\ref{thm:KKP17}.
Moreover, $(f-f_j)_\limplus \in \VspoOm$ by Corollary~\ref{cor-police}
since $0\leq (f-f_j)_\limplus \leq (f-Hf)_\limplus \in \VspoOm$, 
by Lemma~\ref{lem-Vspo-limplus}.
It follows from the comparison principle for obstacle problems
 (Theorem~\ref{thm-comp-obs}) that
$u_j \ge Hf =Hg$.
Furthermore, if $x_0 \in \bdy \Om \setm E'$, then 
$x_0$ is Sobolev regular and thus 
$\liminf_{\Om\ni x \to x_0} Hg(x) = g(x_0)$, 
by the definition of Sobolev regularity if $g \in C(\Rn)$ 
and by Theorem~\ref{thm-Sobolev-reg}\ref{a-contx0} if $g$ is bounded 
in $\Rn$ and
continuous at $x_0$.
Hence
\[
\liminf_{\Om\ni x\to x_0} u_j(x) \ge \liminf_{\Om\ni x\to x_0} Hg(x) 
=g(x_0)=f(x_0)
\quad \text{for all } x_0\in \bdy \Om \setm E'.
\]
On the other hand, $u_j \ge f_j \ge \phi_j \ge k+1$ a.e.\ in 
the open set $\Om \cap  \bigcap_{i=j}^{j+k} G_{i}$ and thus 
$u_j \ge k+1$ everywhere therein,
since $u_j$ is lsc-regularized.
Letting $k \to \infty$ therefore shows that
\[
\lim_{\Om\ni x\to x_0} u_j(x) = \infty \quad \text{for all } x_0\in \bdy\Om \cap E'.
\]
Hence
\[
\liminf_{\Om\ni x\to x_0} u_j(x) \ge f(x_0) \quad \text{for all } x_0\in \bdy \Om.
\]
Thus $u_j \in \UU_f$ and
$u_j\ge \uP f$ in $\R^n$. 
Next note that $Hf$ clearly is a solution of the 
$\K_{Hf,Hf}(\Om)$-obstacle problem.
It thus follows from Theorem~\ref{thm-conv-obst-prob}
that  $u_j\to Hf$ a.e.\ in $\Om$.
Hence $Hf\ge \uP f$ a.e.\ in $\Om$.

For arbitrary $g$ and $f=g+h$ we have by the above and the last part of
Theorem~\ref{thm-conv-obst-prob} that
\begin{equation} \label{eq-lim-max-f,m}
\uP f  \le \lim_{m\to-\infty} \uP \max\{f,m\}
 \le \lim_{m\to-\infty} H \max\{f,m\}
= Hf  \quad \text{a.e.\ in }  \Om.
\end{equation}
By Theorem~\ref{thm-Perron}, $\uP f$ is an 
$\eR$-continuous
function in $\Om$.
On the other hand, $Hf$ is $\LL$-harmonic and thus continuous in $\Om$. 
Hence the inequality $\uP f \le Hf$
holds everywhere in $\Om$.
Applying this
also to $-f$, and using \eqref{eq-lP<=uP}, shows that
\[
\uP f \le Hf = -H(-f) \le -\uP(-f) = \lP f \le \uP f
\quad \text{in } \Om. 
\]
On the other hand $P f= f=Hf$ in $\Omc$.
It also follows (upon letting $h\equiv 0$) that $Pg=Hg$ in $\Rn$.
Moreover, $Pg=Hg=Hf=Pf$ in $\Om$, while
$Pf=f=g+h=Pg+h$ in $\Omc$.
\end{proof}

We are now ready to prove Theorem~\ref{thm-Perron-res-cont}.
When $k(x,y) = |x-y|^{-n-sp}$ and $h \equiv 0$ this was shown for bounded 
$g \in C(\R^n)$, such that $\lim_{x \to \infty} g(x)$ exists,
in Lindgren--Lindqvist~\cite[Theorem~1]{LL17}.

\begin {proof}[Proof of Theorem~\ref{thm-Perron-res-cont}]
By Tietze's extension theorem, we can extend $g$ so that $g \in C(\clOm)$.

Let $\eps>0$. 
Then there is $\gt\in\Lip(\clOm)$ such that $|g-\gt|<\eps$ in $\clOm$.
Let $L$ be the Lipschitz constant of $\gt$.
For each $x\in\clOm$ there is $0<r_x\le \eps/L$ such that
$|g-g(x)|<\eps$ in $B(x,r_x)$.
By compactness, 
\[
\clOm \subset \Om_\eps := \bigcup_{i=1}^N B(x_i,r_{x_i}).
\]
Using the McShane extension, see
e.g.\  Heinonen~\cite[Theorem~6.2]{Heinonen},
$\gt$ can be extended from $\clOm$ to an $L$-Lipschitz function (also
denoted by $\gt$) on $\Om_\eps$.
Then for all $x\in \Om_\eps\setm \Om$ and  $x_i$ 
such that $x\in B(x_i,r_{x_i})$,
\[
|g(x) - \gt(x)| \le |g(x) - g(x_i)| + |g(x_i) - \gt(x_i)| + |\gt(x_i) - \gt(x)|
< 3\eps.
\]
Let also $\gt=g$ in $\Om_\eps^c$.
Then $|g-\gt|<3\eps$ in $\R^n$.
Moreover, $\gt \in  \Vsp(\Om)$.
Indeed, there is $\de>0$ such that $|x-y|\ge\de$ for all
$x\in\Om$ and $y\in\Om_\eps^c$.
Hence, because $\Om$ is bounded and $\gt$ is Lipschitz in $\Om_\eps$
and measurable in $\Rn$,
\begin{align*}
[\gt]_{V^{s, p}(\Omega)}^p
&\le \int_{\Om} \int_{\Om_\eps} \frac{|\gt(x)-\gt(y)|^p}{|x-y|^{n+sp}} \,dy \,dx 
          + \int_{\Om} (2 \|\gt\|_{L^\infty(\R^n)})^p \int_{B(x,\de)^c} \frac{dy \,dx}{|x-y|^{n+sp}}   \\
 &\le \int_{\Om} \int_{\Om_\eps} \frac{(L|x-y|)^p}{|x-y|^{n+sp}} \,dy \,dx 
          + \int_{\Om} (2 \|\gt\|_{L^\infty(\R^n)})^p C\de^{-sp} \,dx< \infty.
\end{align*}

Thus for each $j=1,2,\dots$\,, there are $g_j\in  \Vsp(\Om)$ such that 
$|g-g_j|<2^{-j}$ in $\R^n$.
Extend $h$ so that $h(x)=0$ for $x \in \Om$.
By Theorem~\ref{thm-Hg=Pg}, $g_j$ and $f_j:=g_j+h$ are resolutive and 
$P f_j = Pg_j + h\chi_{\Om^c}$ in $\R^n$.
Using \eqref{eq-lP<=uP} we get, with $f=g+h$,
\[
P f_j -2^{-j} \le  \lP f \le  \uP f   
\le P f_j + 2^{-j} 
\quad \text{in } \R^n.
\]
Similar inequalities hold with $f$ and $f_j$ replaced by $g$ and $g_j$.
Letting $j\to\infty$ then shows that 
\[
 Pf  =\lim_{j \to \infty} P f_j  =\lim_{j \to \infty} P g_j 
+  h\chi_{\Om^c} = Pg  + h\chi_{\Om^c}
\quad \text{in } \R^n.
\]
Since $Pf$ and $Pg$ are bounded in $\Om$, it follows
from Theorem~\ref{thm-Perron-meas} that they are $\LL$-harmonic in $\Om$.
\end{proof}

\begin{remark}
If \ref{P-b} in Definition~\ref{def-Perron} were replaced by
\begin{enumerate}
\renewcommand{\theenumi}{\textup{(\alph{enumi}$'$)}}%
\setcounter{enumi}{\value{saveenumi}}
\item \label{P-b'}
$u$ is bounded from below in $\Rn$,
\end{enumerate}
we would immediately get the equality
\begin{equation*}
\uP g  =  \lim_{m\to-\infty} \uP \max\{g,m\}.
\end{equation*}
With \ref{P-b} in Definition~\ref{def-Perron} we only have a trivial inequality, which however
is enough for our purposes.
(It was used in~\eqref{eq-lim-max-f,m}.)

On the other hand, with \ref{P-b'} we would have  to restrict
Theorem~\ref{thm-Perron-meas} to functions $g \in L^{p-1}_{sp}(\Omega^c)$
that are in addition bounded from below.
This is the main reason for our choice of~\ref{P-b}.
\end{remark}

\section{Perron regularity}\label{sec-perron-reg}

Now we are ready to show that Sobolev regularity is equivalent to Perron regularity.
This also enables us to show that the existence of a barrier characterizes regularity.
Along the way we
show that regularity for obstacle problems is the same as
for the Dirichlet problem, see \ref{b-obst} in Theorem~\ref{thm-main-reg-Perron} below.
Here is the definition of barriers.

\begin{deff}
A function $u:\R^n \to [0,\infty]$ is a \emph{barrier} (relative to $\Omega$) at $x_0\in\bdy\Om$ if
\begin{enumerate}
\item
$u$ is $\LL$-superharmonic in $\Om$,
\item
$\lim_{\Om \ni x \to x_0} u(x)=0$,
\item 
$\inf_{B(x_0,r)^c} u > 0$ for all $r>0$.
\end{enumerate}
\end{deff}

We will also need the following fact.

\begin{lem} \label{lem-obst-KL}
Let $x_0 \in \bdy \Om$ and $d_{x_0}(x):=\min\{1,|x-x_0|\}$.
Then the $\K_{d_{x_0},d_{x_0}}(\Om)$-obstacle problem has a unique solution $u$
which is  continuous in $\Om$.
Moreover 
$u$ is $\LL$-superharmonic in $\Om$ and 
$\LL$-harmonic in 
$G=\{x \in \Om : u(x) > d_{x_0}(x)\}$.
\end{lem}

\begin{proof}
By Theorem~\ref{thm-obst-KL-solv}, there is a 
unique lsc-regularized solution $u$ of the $\K_{d_{x_0},d_{x_0}}(\Om)$-obstacle problem.
That $u$ is continuous in $\Om$ follows directly from Theorem~4.12
in Kim--Lee~\cite{KL}, 
whereas Corollary~4.10 therein shows that $u$ is $\LL$-harmonic in~$G$.
By Theorem~\ref{thm-obst-KL-solv}, $u$ is a supersolution and since it is
lsc-regularized it is $\LL$-superharmonic by Theorem~\ref{thm:KKP17}.
\end{proof}

\begin{thm}\label{thm-main-reg-Perron}
Let $x_0 \in \bdy \Om$ and $d_{x_0}(x):=\min\{1,|x-x_0|\}$.
Then the following are equivalent\/\textup{:}
\begin{enumerate}
\item \label{b-reg}
  $x_0$ is Sobolev regular.
\item \label{b-bdd}
$x_0$ is Perron regular, i.e.
  \[
    \lim_{\Om \ni x \to x_0} P g(x)=g(x_0)
  \]  
for every bounded $g \in C(\Omc)$. 
\item \label{b-contx0}
  \[
    \lim_{\Om \ni x \to x_0} \uP g(x)=g(x_0)
  \]  
for every bounded $ g:\Omc\to \R$ that 
is continuous at $x_0$.
\item \label{b-d}
  \[
    \lim_{\Om \ni x \to x_0} Pd_{x_0}(x)=0.
  \]  
\item \label{b-barrier}
There is a barrier  at $x_0$.
\item \label{b-obst-d}
The continuous solution $u$ of the 
$\K_{d_{x_0},d_{x_0}}(\Om)$-obstacle problem, provided by Lemma~\ref{lem-obst-KL},
satisfies
\[
    \lim_{\Om \ni x \to x_0} u(x)=0.
\]
In particular, this $u$ is a continuous barrier at $x_0$.
\item \label{b-obst}
For every bounded $g \in \VspOm$ 
that is
continuous at $x_0$ and for every
$\psi:\Om\to [-\infty,\infty)$, which is bounded from above and
such that $\K_{\psi,g}(\Om) \ne \emptyset$
and $\limsup_{x \to x_0} \psi(x) \le g(x_0)$,
it is true that the lsc-regularized solution $u$ of the $\K_{\psi,g}(\Om)$-obstacle  problem
satisfies
\[
    \lim_{\Om \ni x \to x_0} u(x)=g(x_0).
\]
\end{enumerate}  
\end{thm}

\begin{proof}
\ref{b-bdd}\imp\ref{b-d}
This is trivial.

\ref{b-d}\imp\ref{b-contx0}
We may assume that $g(x_0)=0$. Let $\eps >0$.
Then there is $m>0$ such that 
$g \le md_{x_0}+\eps$ in $\Omc$.
By definition,
\[
  \limsup_{\Om \ni x \to x_0} \uP g(x) 
  \le     \lim_{\Om \ni x \to x_0} m P d_{x_0}(x)  + \eps=\eps.
  \]
Letting $\eps \to 0$ shows that 
\[
 \limsup_{\Om \ni x \to x_0} \uP g(x) \le 0.
\]
Applying this to $-g$, and also using \eqref{eq-lP<=uP},  shows that 
\[
  \liminf_{\Om \ni x \to x_0} \uP g(x)
  \ge \liminf_{\Om \ni x \to x_0} \lP g(x) \ge  0.
\]

\ref{b-contx0}\imp\ref{b-bdd}
This is immediate since every $g$ satisfying \ref{b-bdd} is resolutive,
by Theorem~\ref{thm-Perron-res-cont}.

\ref{b-reg}\eqv\ref{b-d}
Since $Pd_{x_0}=Hd_{x_0}$, by Theorem~\ref{thm-Hg=Pg}, this
follows directly from
the equivalence
\ref{a-unbdd}\eqv\ref{a-d}
in Theorem~\ref{thm-Sobolev-reg}.

\ref{b-reg}\imp\ref{b-obst-d} 
By Lemma~\ref{lem-obst-KL}, 
$u$ is $\LL$-superharmonic in $\Om$ and $\LL$-harmonic in 
\[
G:=\{x \in \Om : u(x) > d_{x_0}(x)\},
\]
which is an open set.
It follows 
that $u=H_G d_{x_0}$. 

If $x_0 \in \bdy G$, then it follows 
from Corollary~\ref{cor-local-subset} 
that $x_0$ is Sobolev regular also with respect to $G$.
Hence
\begin{equation*}
     \lim_{G \ni x \to x_0} u(x) = d_{x_0}(x_0)=0
     \quad \text{if } x_0 \in \bdy G. 
\end{equation*}
On the other hand, 
\begin{equation*}
     \lim_{\Om \setm G \ni x \to x_0} u(x) 
  = \lim_{\Om \setm G \ni x \to x_0} d_{x_0}(x) =   d_{x_0}(x_0)=0
     \quad \text{if } x_0 \in \bdy (\Om \setm G). 
\end{equation*}
Thus,
\[
     \lim_{\Om \ni x \to x_0} u(x) = 0.
\]

Moreover, by continuity,
\begin{equation*}
\inf_{B(x_0,r)^c} u \ge \inf_{B(x_0,r)^c} d_{x_0} = \min\{r,1\} > 0 \quad \text{for all }r>0.
\end{equation*}
Thus $u$ is a continuous barrier at $x_0$.

\ref{b-obst-d}\imp\ref{b-barrier} 
This is trivial.

\ref{b-barrier}\imp\ref{b-d}
Let $\eps>0$.
Let $u$ be a barrier at $x_0$ and let $\de= \inf_{B(x_0,\eps)^c} u>0$.
Then $\eps + u/\de \in \UU_{d_{x_0}}(\Om)$ and hence 
\[
0 \le \limsup_{\Om\ni x\to x_0} P d_{x_0} (x)  
\le \eps + \lim_{\Om\ni x\to x_0}u(x)/\de = \eps.
\]
Letting $\eps\to0$ concludes the proof.

\ref{b-obst-d}\imp\ref{b-obst} 
We may assume that $g(x_0)=0$.
Let $\eps>0$.
Then there is $m>0$ such that  $g \le md_{x_0}+\eps$ in $\R^n$ and 
$\psi \le md_{x_0} + \eps$ in $\Om$.
Let $v$ be 
the continuous solution of the $\K_{d_{x_0},d_{x_0}}(\Om)$-obstacle problem.
Since $(g-md_{x_0}-\eps)_\limplus \in V^{s, p}_0(\Om)$,
the comparison principle (Theorem~\ref{thm-comp-obs}) implies that
$u \le mv +\eps$ in $\Om$ and thus
by \ref{b-obst-d},
\[
    \limsup_{\Om \ni x \to x_0} u(x) 
   \le m   \lim_{\Om \ni x \to x_0} v(x) +\eps 
   = \eps.
\]
Letting $\eps \to0$ shows that 
\[
    \limsup_{\Om \ni x \to x_0} u(x) \le 0.
\]
On the other hand, 
it also follows from the comparison principle (Theorem~\ref{thm-comp-obs})
that  $u \ge H g$.
Since we have already shown that \ref{b-obst-d}\imp\ref{b-reg} it follows
from 
Theorem~\ref{thm-Sobolev-reg} that
\[
    \liminf_{\Om \ni x \to x_0} u(x) \ge 
    \liminf_{\Om \ni x \to x_0} Hg(x) =0.
\]

\ref{b-obst}\imp\ref{b-obst-d}
This is trivial.
\end{proof}

As an application, we can now prove
Theorem~\ref{BBS2-thm-intro}.

\begin{proof}[Proof of Theorem~\ref{BBS2-thm-intro}]
By the Kellogg property (Theorem~\ref{thm-kellogg})
together with Theorems~\ref{thm-Perron-res-cont} and~\ref{thm-main-reg-Perron},
$u = P g$ fulfills \eqref{eq-BBS2-thm}, which shows the existence.

For the uniqueness we proceed as follows.
By adding a sufficiently
large constant to both $g$ and $u$,
and then rescaling them 
simultaneously we may assume
without loss of generality that $0 \le u \le 1$ and
$0\le g\le 1$. 
Hence $u \in \UU_{g-\chi_E}$ and $u \in \LL_{g+\chi_E}$,
where $E$ is the exceptional set in~\eqref{eq-BBS2-thm}.
Therefore,
by Theorem~\ref{thm-Perron-res-cont}, we see that 
\[
u \ge \uP (g-\chi_E)=P g = \lP (g+\chi_E) \ge u
 \quad \text{in } \Om.
\]
On the other hand, $u=g=Pg$ outside $\Om$.
\end{proof}

We can also deduce the following result.

\begin{prop}
Let $g \in \VspOm$ be bounded and such that 
$g|_\Omc$ is continuous at all points in $\bdy \Om \setm E$,
where $\Csp(E)=0$.
Then there exists a unique bounded $\LL$-harmonic function $u$ in $\Om$
such that $u \equiv g$ outside $\Om$ and 
\begin{equation} \label{eq-BBS2-thm-2}
         \lim_{\Om \ni x\to x_0} u(x) = g(x_0)
         \quad \text{for quasievery } x_0 \in \bdy \Om.
\end{equation}
Moreover $u=P g=Hg$.
\end{prop}

\begin{proof}
By Theorem~\ref{thm-Hg=Pg}, $g$ is resolutive. 
Moreover, by the Kellogg property (Theorem~\ref{thm-kellogg}),
the set $I_\Om$ of Sobolev irregular points satisfies $\Csp(I_\Om)=0$.
On the other hand, by Theorem~\ref{thm-main-reg-Perron}\ref{b-contx0},
$u = Pg$ fulfills \eqref{eq-BBS2-thm-2} for each 
$x_0 \in \bdy \Om \setm (E \cup I)$.
This yields the existence.

The uniqueness is now proved almost verbatim as in the proof of 
Theorem~\ref{BBS2-thm-intro}, but using
Theorem~\ref{thm-Hg=Pg} instead of
Theorem~\ref{thm-Perron-res-cont}.
\end{proof}

\begin{remark}
Depending on the value of $sp$ we have three different cases.

{\bf Case 1.} $sp>n$.
In this case every point has positive capacity, by Lemma~\ref{lem-sp<=n}.
It thus follows from the Kellogg property (Theorem~\ref{thm-kellogg}) 
that every nonempty bounded open set is regular (i.e.\
all its boundary points are regular).
For Sobolev regularity this also follows directly from
the Wiener criterion (Theorem~\ref{thm-Wiener})
together with Lemma~2.17
in Kim--Lee--Lee~\cite{KLL23},
but for Perron regularity this does not seem to be so straightforward without
appealing to Theorem~\ref{thm-main-reg}.

{\bf Case 2.} $1 < sp \le n$.
It follows from Corollary~5.1.14 in Adams--Hedberg~\cite{AH}
that $\Csp(\bdy \Om)>0$, and hence $\Om$ has some regular boundary 
points by the Kellogg property (Theorem~\ref{thm-kellogg}). 
In fact the set of regular boundary points must be dense
in $\bdy \overline{\Om}$.
(Note that by Triebel~\cite[Theorem p.~172 and Remark~4 pp.~189--190]{Triebel95}
and \cite[Corollary~2.6.8 and Proposition~4.4.4]{AH}
the Bessel potential capacity (denoted
by $C_{s,p}$ in \cite{AH}) is comparable to our Sobolev capacity $\Csp$.)

{\bf Case 3.} $sp \le 1$.
In this case it can happen that $\Csp(\bdy \Om)=0$, 
by Theorem~5.1.9 in~\cite{AH},
in which case the Kellogg property does not yield anything.
Nevertheless, the set of regular boundary points is dense
in $\bdy \overline{\Om}$ also in this case.
Indeed, let $x_0 \in \bdy \overline{\Om}$ and $\eps>0$.
Then there is $y \in \Rn \setm \overline{\Om}$ with $|y-x_0|<\eps$.
Now pick a closest point $x \in \bdy\overline{\Om}$ to $y$, i.e.\
$|x-y|=\min_{z \in \bdy\overline{\Om}} |z-y|$. (This point does not have to be unique.)
Then $x$ satisfies the exterior cone condition and is thus regular
by the Wiener criterion (Theorem~\ref{thm-Wiener}).
Moreover, 
\[
    |x-x_0| \le |x-y| + |y-x_0| \le |x_0-y| + |y-x_0| < 2 \eps.
\]
\end{remark}

\end{document}